\documentclass[10pt]{article}

\usepackage{color}
\usepackage{amsmath}

\usepackage{footnote}
\usepackage{amsfonts,color}
\usepackage{amsmath,amssymb}
\usepackage{amssymb} 
\usepackage{mathtools}
\usepackage{amsthm,wasysym}
\usepackage{framed}
\usepackage{cancel}
\usepackage{graphicx}
\usepackage{caption}
\usepackage{subcaption}
\usepackage{rotating}

\usepackage[english]{babel}
\usepackage[utf8]{inputenc}
\makeatletter
\usepackage{float}
\usepackage{wrapfig}
\restylefloat{figure}

\makeatletter
\let\@fnsymbol\@arabic
\makeatother

\usepackage{bbm}

\newcommand{\id}{{ {\mathbbm{1}}}}

\newcommand{\tr}{{\rm tr}}
\newcommand{\dev}{{\rm dev}}
\newcommand{\Sym}{{\rm Sym}}
\newcommand{\sym}{{\rm sym}}
\newcommand{\skw}{{\rm skew}}

\newcommand{\Curl}{{\rm Curl}}
\newcommand{\Div}{{\rm Div}}

\newcommand{\anti}{{\textrm{Anti}}}
\newcommand{\so}{\mathfrak{so}}

\newcommand{\norm}[1]{\|#1\|}

\def\dd{\displaystyle}

\newtheorem{theorem}{Theorem}[section]
\newtheorem{lemma}[theorem]{Lemma}
\newtheorem{remark}[theorem]{Remark}
\newtheorem{proposition}[theorem]{Proposition}

\newtheorem{definition}[theorem]{Definition}
\def\barr{\begin{array}}
\usepackage{lscape}

	\def\earr{\end{array}}
\def\bec#1{\begin{equation}\label{#1}}
\def\becn{\begin{equation*}}
\def\endec{\end{equation}}
\def\endecn{\end{equation*}}
\def\dd{\displaystyle}
\def\bfm#1{ ox{\boldmat}}

\title{Rayleigh waves  in isotropic  elastic  materials with micro-voids }

\author{Emilian Bulgariu\thanks{ Emilian Bulgariu,\ Department of Exact Sciences, Faculty of Horticulture,   Ia\c si University of Life Sciences (IULS),  Aleea Mihail Sadoveanu no.3, 700490 Ia\c si,
		Romania, email:  ebulgariu@uaiasi.ro}  \quad and \quad  Ionel-Dumitrel Ghiba\thanks{ Ionel-Dumitrel Ghiba, \ Department of Mathematics,  Alexandru Ioan Cuza University of Ia\c si,  Blvd.
		Carol I, no. 11, 700506 Ia\c si,
		Romania; and  Octav Mayer Institute of Mathematics of the
		Romanian Academy, Ia\c si Branch,  700505 Ia\c si, email:  dumitrel.ghiba@uaic.ro} \quad  and \quad  Hassam Khan\thanks{Hassam Khan,    Department of Mathematics and Statistics,  Institute of Business and Management Sciences (IOBM), Karachi. Pakistan, email: dr.hassam@iobm.edu.pk}
\\	and \quad  Patrizio Neff\,\thanks{Patrizio Neff,  \ \ Head of Lehrstuhl f\"{u}r Nichtlineare Analysis und Modellierung, Fakult\"{a}t f\"{u}r
	Mathematik, Universit\"{a}t Duisburg-Essen,  Thea-Leymann Str. 9, 45127 Essen, Germany, email: patrizio.neff@uni-due.de} }
\date{\it Dedicated to Yibin Fu on the occasion of his 60th birthday}
\begin{document}
	\maketitle
	\pagenumbering{}	
	
	\pagenumbering{arabic}	
	
\begin{abstract}	
	In this paper, we show that a general method introduced by Fu and Mielke allows to give a complete answer on the existence and uniqueness of a subsonic solution   describing  the propagation of  surface waves in an isotropic half space modelled with the linear   theory of isotropic elastic materials with  micro-voids. Our result is valid for the entire class of materials admitting real wave propagation which include   auxetic materials (negative Poisson's ration)  and  composite materials with negative-stiffness inclusions  (negative Young's modulus). Moreover, the used method allows to formulate a simple and complete numerical strategy for the computation of the solution.  
	
	\medskip
  
  \noindent\textbf{Keywords:} relaxed micromorphic model,  materials with micro-voids,  Riccati equation, Rayleigh waves, secular equation, existence and uniqueness, micro-voids model, Fu and Mielke's method.
  
  \medskip
  
  \noindent\textbf{AMS 2020 MSC:} 74J15, 74M25, 74H05, 74J05
	
\end{abstract}

\begin{footnotesize}
	\tableofcontents
\end{footnotesize}

\section{Introduction}

Starting with \cite{NeffGhibaMicroModel} and \cite{MadeoNeffGhibaW} we have re-investigated the general micromorphic model \cite{Eringen99,Mindlin64} and we have introduced the relaxed micromorphic model which is general enough to incorporate all known subclasses of micromorphic models (Cosserat, microstretch, micro-voids, etc) but in the same time most simple to be considered in practice and for fitting the constitutive parameters. The relaxed micromorphic model can be used for    engineering metamaterials design and metamaterial shields for inner protection and outer tuning \cite{demore2022unfolding,rizzi2022towards,ramirez2023multi,voss2022modeling,rizzi2022boundary,rizzi2022metamaterial}. In the micromorphic model $P=(P_{ij}):\Omega\times [0,T]\rightarrow \mathbb{R}^{3 \times 3}$  describes the substructure of the material which can rotate, stretch, shear and shrink, while $u=(u_i) :\Omega\times [0,T]\rightarrow  \mathbb{R}^3$  is the displacement of the macroscopic material points. Due to the coercive inequalities proved by Neff, Pauly and  Witsch \cite{NeffPaulyWitsch,NPW2,NPW3} and by Bauer, Neff, Pauly and Starke \cite{BNPS2,BNPS3} (see also \cite{LNPzamp2013}) we have shown \cite{GhibaNeffExistence,NeffGhibaMadeoLazar} that the  relaxed linear micromorphic model is well-posed for  non-symmetric symmetric Cauchy force stresses (non-vanishing Cosserat couple modulus) as well as for the important case of symmetric Cauchy force stresses (zero Cosserat couple modulus), even if the internal energy density is not positive definite in terms of the strain tensors previously used in the literature. The new inequalities given by Lewintan, M\"uller and Neff \cite{lewintan2021korn,lewintan2021nevcas} and by Gmeineder, Lewintan, 
 and Neff  \cite{gmeineder2023optimal,gmeineder2022korn} lead to existence results even in more degenerated cases (less  partial derivatives involved in the model).

When the micromorphic  theory is linked to the theory of dislocations, there are many reasons to assume that the Cosserat couple modulus vanishes and the Cauchy force stresses are symmetric, see the paper by Kr\"oner and \cite{NeffGhibaMicroModel}. The particular relaxed micromophic model  constructed for $P=\zeta \cdot \id$ will be called \textit{micro-voids model} in the rest of the paper. We show that the obtained initial-boundary value problem is identical to the micro-voids theory given by  Goodman and Cowin \cite{goodman1972continuum} and by Nunziato and Cowin  \cite{NunziatoCowin79,CowinNunziato83,IesanVoid,IesanCarteElast}. Therefore, from two different perspectives and approaches we are arriving at the same 
 total energy of the model (see \cite{IesanCarteElast,GhibaZAMM13})  given by\footnote{For $ a,b\in \mathbb{R}^{3 \times 3} $ we let $\langle{a,b}\rangle_{\mathbb{R}^3}$ denote the scalar product on $ \mathbb{R}^3$ with associated vector norm $\norm{a}^2=\langle {a,a}\rangle$. We denote by $\mathbb{R}^{3\times 3}$   the set of real $3 \times 3$ second order tensors, written with capital letters. Matrices will be denoted by bold symbols, e.g. ${X}\in \mathbb{R}^{3\times 3}$, while $X_{ij}$ will denote its component. The standard Euclidean product on $\mathbb{R}^{3 \times 3}$ is given by $\langle{ {X},{Y}}\rangle_{\mathbb{R}^{3 \times 3}}=\tr( {X}\,{Y}^T)$, and thus, the Frobenious tensor norm is $\norm{ {X}}^2=\langle{ {X}, {X}}\rangle_{\mathbb{R}^{3 \times 3}}$. In the following we omit the index $\mathbb{R}^3, \mathbb{R}^{3\times 3}$. The identity tensor on $\mathbb{R}^{3\times 3}$ will be denoted by $\id$, so that $\tr( {X})=\langle{ {X},\id}\rangle$. We let $\Sym$  denote the set of symmetric tensors. We adopt the usual abbreviations of  Lie-algebra theory, i.e., $\mathfrak{so}(3):=\{ {A}\,\in \mathbb{R}^{3\times 3}| {A}^T=- {A}\}$ is the Lie-algebra of skew-symmetric tensors and $\mathfrak{sl}(3):=\{ {X}\, \in \mathbb{R}^{3\times 3} |\tr( {X})=0 \}$  is the Lie-algebra of traceless tensors. For all $ {X} \in\mathbb{R}^{3\times 3}$  we set $\sym\,  {X}=\frac{1}{2}( {X}^T+ {X}) \in\Sym,\,\skw  {X}=\frac{1}{2}( {X}- {X}^T)\in \mathfrak{so}(3)$ and the deviatoric (trace-free) part $\dev\, {X}= {X}-\frac{1}{3}\tr( {X})\in\, \mathfrak{sl}(3)$ and we have the orthogonal Cartan-decomposition of the Lie-algebra $
 	\mathfrak{gl}(3)=\{\mathfrak{sl}(3)\cap \Sym(3)\}\oplus\mathfrak{so}(3)\oplus\mathbb{R}\cdot\id,\ 
 	{X}= \dev\,\sym\, {X}+\skw\, {X}+\frac{1}{3}\tr( {X})\,\id.\
 	$} 
\begin{align}\label{energyvoid}
\mathcal{L}  =\int_\Omega\bigg(&\frac{1}{2}\varrho_0\,\|u_{,t}\|^2+\frac{1}{2}\varrho_0\,\varkappa\,\|\zeta_{,t}\|^2+\mu_{\rm e}\,\| \sym \,{\rm D}u \|^2+ \frac{\lambda_{\rm e}}{2}\, [\tr({\rm D}u)]^2+\frac{\xi }{2}\,\zeta^2+\frac{\alpha}{2}\,\|{\rm D}\zeta\|^2+\beta \,\tr({\rm D}u)\,\zeta\bigg)\, dv,
\end{align}
with $\mu_{\rm e},\lambda_{\rm e},\xi  ,\alpha$, $\beta $ the constant constitutive coefficients and $\varkappa$ a micro-inertia coefficient. 
However, from our modelling process, we are able to observe that the coefficients $\xi $ and $\beta $ from above are defined by the elastic bulk modulus $\kappa_{\rm  e}=\frac{2\mu_{\rm e}+3\lambda_{\rm e}}{3}$ and the microscopic bulk modulus $\kappa_{\rm  micro}$ through $
\beta =-3\kappa_{\rm e},$ 
$\xi =9(\kappa_{\rm e}+\kappa_{\rm micro})$. When the  microscopic bulk modulus $\kappa_{\rm  micro}\to \infty$, see Section \ref{subsC}),  $\xi \to \infty$, while $\mu_{\rm e} $ and $ \lambda_{\rm e} $ tend to the macroscopic classical Lam\'e coefficients and we recover classical linear elasticity. Moreover, in this limit case, from the boundedness of the internal energy density it follows that $\zeta\to 0$ and the total energy is that from classical linear elasticity.
 
 When it comes to understanding how waves travel on surfaces, a significant focus is placed on studying seismic waves. Such waves, crucial for modelling the waves generated by earthquakes on Earth's surface, are known as Rayleigh waves \cite{Rayleigh}, characterized by a mixture of horizontal (longitudinal) and vertical (transverse) movements. The horizontal displacement components align with the direction of wave propagation, while the vertical components extend into the surrounding half-space. The investigation of Rayleigh waves has garnered considerable interest among scientists due to their practical applications in fields such as seismology and near-surface geophysical exploration.

There are several methods available for developing solutions that describe the propagation of Rayleigh waves in classical and generalized models of elastic materials. For example, Hayes (referenced in \cite{hayes1986inhomogeneous, boulanger1993bivectors}) introduced the "directional-ellipse method" to systematically study all inhomogeneous plane waves that can propagate in classical linear elasticity. In classical elasticity of isotropic materials, the solution for the Rayleigh surface wave speed, also known as the "secular equation," has been explored numerically (as seen in references like \cite{Rayleigh, hayes1962note}). Hayes and Rivlin also delved into the secular equation \cite{HayesRivlin}. For a more comprehensive understanding of this problem, additional details can be found in the book \cite{Achenbach} and other works like \cite{Nkemzi, Nkemzi2, Rahman, Malischewsky, Li, destrade2007seismic, Ting2, Ting, VinhMalischewsky1, VinhMalischewsky2, VinhOgden}.
Another method, known as the "Stroh formalism," was introduced in \cite{stroh1958dislocations}. Stroh, in \cite{stroh1962steady}, simplified the complex secular equation originally derived by Synge \cite{synge1956elastic}, providing a real expression for it. Currie's work \cite{currie1974rayleigh} also addresses this. In classical elasticity, the Stroh formalism has successfully solved the problem of finding explicit secular equations for certain types of material, such as anisotropic and inhomogeneous materials. Relevant references for this include \cite{taziev1989dispersion, mozhaev1995some, destrade2001explicit, destrade2007seismic, destrade2002incompressible, ting2002explicit,lazar2005cosserat}.
However, in the context of generalized models, the Stroh formalism does not yield an explicit analytical form of the secular equation.

It is important to note that in the process of constructing the solution, including the displacement and other unknown variables in the generalized model, all methods impose certain restrictions, which form an admissible set for the Rayleigh wave solution. These restrictions must be satisfied by the Rayleigh surface wave speed, which is the solution of the secular equation. For instance, in the case of classical isotropic materials, the Rayleigh surface wave speed ($v_s$) must adhere to the inequality $0 \leq v < v_s = \sqrt{\frac{\lambda_{\rm e}+2\mu_{\rm e}}{\varrho_0}}$.

In some generalized continuum models, the secular equation, which provides the surface wave speed, is still not explicitly derived. Even when the secular equation is available in explicit analytical form, authors often do not demonstrate the existence and uniqueness of a solution within the admissible set. In most cases, at best, this critical aspect is only speculated upon or numerically explored for a few specific materials \cite{ChiritaGhiba2,Chirita2013,ChiritaDanescu2015,Bucur2016,GalesChirita2020}.

Regarding the model discussed in this paper, the existence of an admissible solution has been established in \cite{ChiritaGhiba2} under certain hypotheses, i.e., \begin{equation}
\label{auxasump0}{v}_{t}^{2}\leq {v}_{s}^{2}-\displaystyle\frac{{\beta }^{2}}%
{{\varrho_0 }\,\xi } \ \ \text{ and } \ \ {v}_{t}^{2}
\leq\frac{\alpha}{{\varrho_0 } \,\varkappa}, \quad \text{where}\quad {v}_t=\sqrt{\frac{\mu_{\rm e}}{\varrho_0 }}.
\end{equation} However, the uniqueness of such an admissible solution has not been addressed.

It is important to note that the conditions presented in  \eqref{auxasump0} have their limitations. On one hand, when $\alpha$ approaches zero, the constraints in \eqref{auxasump0} indicate that the limit of the group/phase velocity of the acoustic branch of the shear-rotational wave as $\omega$ approaches zero (or as $k$ approaches zero), which is denoted as ${v}_t$, tends to zero. This contradicts the principles of classical elasticity.
On the other hand, when $\xi$ tends to infinity (corresponding to the model converging to classical linear elasticity), the restriction in \eqref{auxasump0}$_1$ implies that the result established in \cite{ChiritaGhiba2} is valid only when ${c}_t^2 \leq {c}_s^2$ (assuming $\mu_{\rm e}>0$ and $\lambda_{\rm e}+\mu_{\rm e}\geq 0$).

To address the shortcomings mentioned earlier, our current study explores an alternative approach for examining Rayleigh waves in elastic materials containing voids. This method was initially introduced by Fu and Mielke \cite{fu2002new,mielke2004uniqueness} and was inspired by the prior work of Mielke and Sprenger \cite{mielke1998quasiconvexity}, which has stronger ties to control theory \cite{knobloch2012topics} than surface wave propagation. Notably, this approach stands out conceptually from other methods and offers a straightforward numerical algorithm for computing Rayleigh wave solutions.
What distinguishes the {\it Fu and Mielke's method} further is that it provides an elegant proof of existence and uniqueness while offering practical computational advantages. It allows us to derive the explicit form of the secular equation without the need for a priori knowledge of analytical expressions (as functions of wave speed) for certain eigenvalues and their associated eigenvectors, which most other methods require. Indeed, calculating these symbolic expressions for eigenvalues and eigenvectors as functions of wave speed poses a primary challenge in many generalized models, with the exception of a few models where this task is relatively straightforward, such as classical isotropic linear elasticity \cite{Achenbach,hayes1962note}, albeit under certain restrictive conditions imposed on the constitutive coefficients.

In a related development, it has been observed in \cite{khan2022existence} that the method initially introduced in \cite{mielke2004uniqueness,fu2002new} within the context of classical elasticity is applicable to generalized elastic materials as well. By leveraging properties of an "algebraic Riccati equation" \cite{mielke2004uniqueness,fu2002new}, we deduce the form of the secular equation for isotropic elastic materials with micro-voids. Additionally, we establish that this secular equation possesses a solution within the admissible set and demonstrate its uniqueness.

Our result is valid for materials characterized  by constitutive parameters satisfying\footnote{In terms of the Young's modulus and the Poisson's ratio these restrictions imply
	$
	(E  <0$ $\text{and}$ $\nu   <-1$ $\text{or}$ $(E  >0$ $\text{and}$ $(-1<\nu   <\frac{1}{2}$  $\text{or}$ $\nu   >1)), $
	while in terms of the constitutive parameters of the relaxed micromorphic model the inequality \eqref{rwc10}$_4$ reads 
	$(\lambda_{\rm e}+2\mu_{\rm e})(\kappa_{\rm e}+\kappa_{\rm micro})>\kappa_{\rm e}^2.
	$
	 }
\begin{align}\label{rwc10}\mu_{\rm e}>0, \qquad \lambda_{\rm e}+2\mu_{\rm e}>0, \qquad \alpha>0,\qquad  (\lambda_{\rm e}+2\mu_{\rm e})\,\xi >\beta ^2.\end{align}
Hence, the approach given in this paper
\begin{itemize}
	\item  is the most comprehensive, similar to the findings in \cite{ChiritaGhiba2}, and results in a straightforward numerical algorithm,
	\item is applicable to the entire spectrum of materials that allow real wave propagation, encompassing.  
	\begin{itemize}
		\item auxetic materials (negative Poisson's ratio) \cite{Lakes93,lakes1987foam, burns1987negative,rothenburg1991microstructure,rothenburg1991microstructure,evans1989microporous, jiang2014negative, brandel2001negative, grima2006negative, yang2012non, friis1988negative},  
		\item composite materials with negative-stiffness inclusions  (negative Young's modulus) \cite{lakes2001extreme}. 
	\end{itemize}
\end{itemize}

\section{The micro-voids model -- a particular  case of the relaxed micromorphic model}\label{subsC}\setcounter{equation}{0}
In this subsection we provide the initial-boundary value problem of a particular  relaxed micromorphic isotropic  model  \cite{NeffGhibaMicroModel,MadeoNeffGhibaW,MadeoNeffGhibaWZAMM,madeo2016reflection,NeffGhibaMadeoLazar} when the micro-distortion tensor $P$ is assumed to a priori  have the special structure ${P}=\zeta\!\!\cdot\!\! \id\in\mathbb{R}^{3 \times 3}$.  Such a micro-distortion tensor is not able to   describes the general micro-rotation, micro-stretching or micro-shearing, but uniform stretching and shrinking of the microstructure (micro-void).  The body occupying the domain $\overline{\Omega}\subset \mathbb{R}^3$ is referred to a fixed system of rectangular Cartesian axes  $Ox_i$ $(i=1,2,3)$, $\{e_1, e_2, e_3\}$ being the unit vectors of these axes. We denote  by $n$ the outward unit normal on $\partial\Omega$.

\subsection{The relaxed micromorphic model}

In the relaxed micromorphic model, in which a positive  Cosserat modulus $\mu_{\rm c} >0$ is related to the isotropic Eringen-Claus  model for dislocation dynamics \cite{Eringen_Claus69,EringenClaus,Eringen_Claus71}, the free energy is given by\footnote{Let us recall some other useful notations for the present work. 
	 For a  regular enough function $f(t,x_1,x_2,x_3)$,  $f_{,t}$ denotes the derivative with respect to the time $t$, while  $ \frac{\partial\, f}{\partial \,x_i}$ denotes the $i$-component of the gradient ${\rm D}f$.  
	For vector fields $u=\left(    u_1, u_2, u_3\right)^T$ with  $u_i\in 
	{\rm H}^1(\Omega)\,=\,\{u_i\in {\rm L}^2(\Omega)\, |\, {\rm D}\, u_i\in {\rm L}^2(\Omega)\}, $  $i=1,2,3$,
	we define
	$
	\mathrm{D} u:=\left(
	{\rm D}\,  u_1\,|\,
	{\rm D}\, u_2\,|\,
	{\rm D}\, u_3
	\right)^T.
	$
	The corresponding Sobolev-space will be also denoted by
	$
	{\rm H}^1(\Omega)$.  In addition, for a tensor field
	${P}$ with rows in ${\rm H}({\rm div}\,; \Omega)$, i.e.,
	$
	{P}=\begin{footnotesize}\begin{footnotesize}\begin{pmatrix}
	{P}^T.e_1\,|\,
	{P}^T.e_2\,|\,
	{P}^T\, e_3
	\end{pmatrix}\end{footnotesize}\end{footnotesize}^T$ with $( {P}^T.e_i)^T\in {\rm H}({\rm div}\,; \Omega):=\,\{v\in {\rm L}^2(\Omega)\, |\, {\rm div}\, v\in {\rm L}^2(\Omega)\}$, $i=1,2,3$,
	we define
	$
	{\rm Div}\, {P}:=\begin{footnotesize}\begin{footnotesize}\begin{pmatrix}
	{\rm div}\, ( {P}^T.e_1)^T\,|\,
	{\rm div}\, ( {P}^T.e_2)^T\,|\,
	{\rm div}\, ( {P}^T\, e_3)^T
	\end{pmatrix}\end{footnotesize}\end{footnotesize}^T
	$
	while for tensor fields $ {P}$ with rows in ${\rm H}({\rm curl}\,; \Omega)$, i.e.,
	$
	{P}=\begin{footnotesize}\begin{footnotesize}\begin{pmatrix}
	{P}^T.e_1\,|\,
	{P}^T.e_2\,|\,
	{P}^T\, e_3
	\end{pmatrix}\end{footnotesize}\end{footnotesize}^T$ with $( {P}^T.e_i)^T\in {\rm H}({\rm curl}\,; \Omega):=\,\{v\in {\rm L}^2(\Omega)\, |\, {\rm curl}\, v\in {\rm L}^2(\Omega)\}
	$, $i=1,2,3$,
	we define
	$
	{\rm Curl}\, {P}:=\begin{footnotesize}\begin{footnotesize}\begin{pmatrix}
	{\rm curl}\, ( {P}^T.e_1)^T\,|\,
	{\rm curl}\, ( {P}^T.e_2)^T\,|\,
	{\rm curl}\, ( {P}^T\, e_3)^T
	\end{pmatrix}\end{footnotesize}\end{footnotesize}^T
	.
	$
}
\begin{align}\label{XXXX}
W_{\rm relax}&=\mu_{\rm e}  \|\sym (\mathrm{D}u -{P})\|^2+\mu_{\rm c} \|\skw(\mathrm{D}u -{P})\|^2+ \frac{\lambda_{\rm e}}{2}\, [\tr(\mathrm{D}u -{P})]^2+\mu_{\rm micro} \|\sym \, {P}\|^2+ \frac{\lambda_{\rm micro}}{2} [\tr({P})]^2\notag\\&
\quad \quad  +\frac{\mu_{\rm e}L_{\rm c}^2}{2}\left[{a_1}\| \dev\,\sym \,\Curl\, {P}\|^2 +{a_2}\| \skw \,\Curl\, {P}\|+ \frac{a_3}{3}\, \tr(\Curl\, {P})^2\right],
\end{align}
where $(\mu_{\rm e},\lambda_{\rm e}) $,  $(\mu_{\rm micro},\lambda_{\rm micro})$,  $\mu_{\rm c}, L_{\rm c} $  and $( a_1, a_2, a_3)$  are the elastic moduli  representing the parameters related to the meso-scale,  the
parameters related to the micro-scale, the Cosserat couple modulus, the characteristic length, and the three
general isotropic curvature parameters (weights), respectively.  Formally,  letting  $L_{\rm c}\to \infty$  means a ``zoom" into the micro-structure while $L_{\rm c}\to 0$ means considering  arbitrary large bodies while retaining the size of the unit-cell or keeping the dimensions of the body fixed while reducing the dimensions of the unit cell to zero, or, in other words, ``no special effects of the microstructure taking into account" (classical elasticity).

The complete system of linear partial differential equations in terms of the kinematical unknowns $u$ and $P$ is given by
\begin{align}\label{eqisoup}
\varrho_0\,u_{,tt}&={\rm Div}[\underbrace{2\,\mu_{\rm e} \, \sym(\mathrm{D}u -{P})+2\,\mu_{\rm c} \, \skw(\mathrm{D}u -{P})+\lambda_{\rm e} \,\tr(\mathrm{D}u -{P}){\cdot} \id}_{\textrm{$=:\!\sigma$ \ the non-symmetric force-stress tensor}}]+f\, ,\\\notag
\varrho_0\, \eta\,\tau_{\rm c}^2\,\,P_{,tt}&=-{\mu_{\rm e}L_{\rm c}^2}\,\Curl [\underbrace{a_1 \, \dev\,\sym \,\Curl\, {P}+a_2\, \skw \,\Curl\, {P} +\frac{a_3}{3}\, \tr(\Curl\, {P}){\cdot} \id}_{=:\,m\,\textrm{ the second-order moment stress  tensor }}]\\&\quad\ +2\,\mu_{\rm e} \, \sym(\mathrm{D}u -{P})+\lambda_{\rm e} \tr(\mathrm{D}u -{P}){\cdot} \id-\underbrace{(2\,\mu_{\rm micro}\, \sym \, {P}+\lambda_{\rm micro} \tr ({P}){\cdot} \id)}_{=:\,\sigma_{\rm micro} \ \text{the symmetric  microstress tensor}}+{M}\,  \ \ \ \text {in}\ \ \  \Omega\times [0,T],\notag
\end{align}
where $f :\Omega\times [0,T]\rightarrow  \mathbb{R}^3$ describes the external body force,  ${M}:\Omega\times [0,T]\rightarrow \mathbb{R}^{3 \times 3}$ describes the external body moment, $\varrho_0$ is the mass density and  $\eta\,\tau_{\rm c}^2$ is the inertia coefficient, with $\eta>0$ a weight parameter and $\tau_{\rm c}$ the internal characteristic time \cite[page 163]{Eringen99}. 

To the system of partial differential equations of this model we adjoin the natural boundary conditions
\begin{align} \label{bcdev}
{u}({x},t)=\widetilde{u}({x},t), \ \ \ \quad \quad {P}_i({x},t)\times n(x) =\widetilde{p}({x},t), \ \ \ i=1,2,3, \ \ \
\ ({x},t)\in\partial \Omega\times [0,T],
\end{align}%
and the nonzero initial conditions
\begin{align}\label{icdev}
&{u}({x},0)={u}_0(
x), \quad\quad\quad \dot{u}({x},0)=\dot{u}_0(
x),\quad\quad\quad {P}({x},0)={P}_0(
x), \quad\quad\quad  \dot{P}({x},0)=\dot{P}_0(
x),\ \ \text{\ \ }{x}\in \bar{\Omega}\,,
\end{align}%
where  $\widetilde{u}$, $\widetilde{p}$, ${u}_0, \dot{u}_0, {P}_0$ and $\dot{P}_0$ are prescribed functions,   satisfying $u_0(x)=\widetilde{u}$ and $P_{0i}(x)\times n(x)=\widetilde{p}$ on $\partial \Omega$.

The  internal energy is positive definite in terms of the independent constitutive variables $\mathrm{D}u -{P}$, $\sym \, {P}$, $\Curl\, {P}$ if and only if \begin{align}\mu_{\rm e}&>0,\qquad \qquad\quad \ \,  \kappa_{\rm e}:=\frac{2\,\mu_{\rm e}+3\,\lambda_{\rm e}}{3}>0,\qquad\qquad\qquad  \ \  \mu_{\rm c}>0,  \\\mu_{\rm micro}&>0, \qquad\qquad \kappa_{\rm micro}:=\frac{2\,\mu_{\rm micro}+3\,\lambda_{\rm micro}}{3}>0,\qquad\qquad  a_1>0, \qquad \qquad a_2>0, \qquad\qquad  a_3> 0.\notag\end{align}
However, existence and uniqueness of the solution is guaranteed for the weaker conditions 
\begin{align}\mu_{\rm e}&>0,\qquad \qquad\quad \ \,  \kappa_{\rm e}:=\frac{2\,\mu_{\rm e}+3\,\lambda_{\rm e}}{3}>0,\qquad\qquad\qquad  \ \  \mu_{\rm c}\boldsymbol{\geq} 0,  \\\mu_{\rm micro}&>0, \qquad\qquad \kappa_{\rm micro}:=\frac{2\,\mu_{\rm micro}+3\,\lambda_{\rm micro}}{3}>0,\qquad\qquad  a_1>0, \qquad \qquad a_2\boldsymbol{\geq} 0, \qquad\qquad  a_3\boldsymbol{\geq}   0,\notag\end{align}
due to the  new inequalities given by Lewintan, M\"uller and Neff \cite{lewintan2021korn,lewintan2021nevcas} and by Gmeineder, Lewintan, 
and Neff  \cite{gmeineder2023optimal,gmeineder2022korn}.

We say that there exists  {\it real plane waves}\footnote{The plane wave is called ``real'' since it is defined by real values of $\omega$.}  in a given  direction $\xi=(\xi_1,\xi_2,\xi_3)$, $\lVert{\xi}\rVert^2=1$,  if for every  wave number  $k$ the system of partial differential equations \eqref{eqisoup} admits a plane wave ansatz solution 
only for  real frequencies $\omega\in \mathbb{R}$.  It is proven in \cite{neff2017real} that the dynamic relaxed micromorphic model \eqref{eqisoup} admits real planar waves if and only if 
\begin{align}\label{relrw}
	\mu_{\rm e}>0,\qquad\mu_{\rm micro}>0,\qquad \mu_{\rm c}\geq0,\qquad \kappa_{\rm e}+\kappa_{\rm micro}>0,\qquad 2\,\mu_{\rm macro}+\lambda_{\rm macro}>0,
\end{align}	 
where
\begin{align}
 \kappa_{\rm macro}:=\frac{\kappa_{\rm e}\,\kappa_{\rm micro}}{\mu_{\rm e}+\kappa_{\rm micro}},\qquad \mu_{\rm macro}:=\frac{\mu_{\rm e}\,\mu_{\rm micro}}{\mu_{\rm e}+\mu_{\rm micro}}, \qquad \lambda_{\rm macro}:=\frac{3\,\kappa_{\rm macro}-2\,\mu_{\rm macro}}{3}.
\end{align}

As remarked in \cite{neff2017real}  if $\mu_{\rm e}+\mu_{\rm micro}>0$ and $\kappa_{\rm e}+\kappa_{\rm micro}>0$, the macroscopic parameters are less or equal than respective microscopic parameters, namely
\begin{align}
\kappa_{\rm micro}&\geq\kappa_{\rm macro},\qquad\qquad \mu_{\rm micro}\geq\mu_{\rm macro},\notag\\
\kappa_{\rm e}&\geq\kappa_{\rm macro},\qquad\qquad\quad\ \  \mu_{\rm e}\geq\mu_{\rm macro},
\end{align} 
while letting  $\mu_{\rm micro}\rightarrow+\infty$ and $\kappa_{\rm micro}\rightarrow+\infty$ (or $\mu_{\rm micro}\rightarrow+\infty$ and $\lambda_{\rm micro}>\mathrm{const}$) generates the limit condition for real wave velocities ($\mu_{\rm e}\rightarrow\mu_{\rm macro}$)
\begin{align}
\mu_{\rm macro}>0,\qquad \mu_{\rm c}\geq0,\qquad 2\,\mu_{\rm macro}+\lambda_{\rm macro}>0\label{Macro}
\end{align}
which coincides, up to the Cosserat couple modulus $\mu_{\rm c}$, with the strong ellipticity condition in isotropic linear elasticity, and it coincides fully with the condition for real wave velocities in Cosserat elasticity \cite{khan2022existence}.

\subsection{The micro-voids  model}\label{voidapp}

In this subsection we consider  a particular relaxed micromorphic model and we show that the obtained system of partial differential equation coincide (after an identification of  parameters) with that from the micro-void model. 

To this aim, let us consider the  particular case of \eqref{eqisoup} in which we assume a priori the kinematical constraint $P=\zeta{\cdot} \id$. 
The boundary condition which follows from the  tangential boundary condition \eqref{bcdev} is then the strong anchoring condition
\begin{align}
\zeta(x,t)=\widetilde{ \zeta}({x},t)\quad \quad \text{on} \quad \partial \Omega\times [0,T].
\end{align}

Note that for all differentiable functions $\zeta:\mathbb{R}\rightarrow\mathbb{R}$ on $\Omega$ we have\footnote{We use the canonical identification of $\mathbb{R}^3$ with $\so(3)$, and, for
	$
	{A}=	\begin{footnotesize}\begin{pmatrix}
	0 &-a_3&a_2\\
	a_3&0& -a_1\\
	-a_2& a_1&0
	\end{pmatrix}\end{footnotesize}\in \so(3)
	$
	we consider the operators $\anti:\mathbb{R}^3\rightarrow \so(3)$ through
	$ (\anti(v))_{ij}=-\epsilon_{ijk}\,v_k, \  \forall \, v\in\mathbb{R}^3$,
	where $\epsilon_{ijk}$ is the totally antisymmetric third order permutation tensor.}
$
\Curl(\zeta{\cdot} \id)=-\anti\, ({\rm D}\zeta)\in\so(3)$
and we obtain
that the total energy can be written as 
\begin{align}
\mathcal{L}(u_{,t},\zeta_{,t},\mathrm{D}u  -\zeta\!\!\cdot\!\! \id, \mathrm{D}\,\zeta)
=\int_\Omega\bigg(&\frac{1}{2}\,\varrho_0\,\|u_{,t}\|^2+\,\frac{1}{2}\,\varrho_0\,\eta\,\tau_{\rm c}^2\,\,\,\|(\zeta\!\!\cdot\!\! \id)_{,t}\|^2\\&+
\mu_{\rm e}  \|\sym (\mathrm{D}u -\zeta\!\!\cdot\!\! \id)\|^2+\mu_{\rm c} \|\skw\, \mathrm{D}u \|^2+ \frac{\lambda_{\rm e}}{2}\, [\tr(\mathrm{D}u -\zeta\!\!\cdot\!\! \id)]^2\notag\\&+\mu_{\rm micro} \|\sym (\zeta\!\!\cdot\!\! \id)\|^2+ \frac{\lambda_{\rm micro}}{2} [\tr(\zeta\!\!\cdot\!\! \id)]^2  +\frac{\mu_{\rm e}L_{\rm c}^2}{2}{a_2}\| \skw \,\Curl(\zeta\!\!\cdot\!\! \id)\|^2\bigg)\,dv.\notag
\end{align}
Since $\|\skw \,\mathrm{D}u \|^2$ alone is  not infinitesimally  frame-invariant, $\mu_{\rm c}$ has to vanish, and the total energy is given by
\begin{align}
\mathcal{L}(u_{,t},\zeta_{,t},&\mathrm{D}u  -\zeta\!\!\cdot\!\! \id, \mathrm{D}\,\zeta)=\int_\Omega\bigg(\frac{1}{2}\,\varrho_0\,\|u_{,t}\|^2+\,\frac{3}{2}\,\varrho_0\,\eta\,\tau_{\rm c}^2\,\zeta_{,t}^2+
\mu_{\rm e}  \|\sym (\mathrm{D}u -\zeta\!\!\cdot\!\! \id)\|^2+ \frac{\lambda_{\rm e}}{2}\, [\tr(\mathrm{D}u -\zeta\!\!\cdot\!\! \id)]^2\notag\\&\qquad \qquad \qquad \qquad \qquad +\frac{3}{2}(2\,\mu_{\rm micro} + 3\,\lambda_{\rm micro})\, \zeta^2  +\mu_{\rm e}L_{\rm c}^2{a_2}\|\mathrm{D}\, \zeta\|^2\bigg)\,dv\,
\notag\\
=\int_\Omega\bigg(&\frac{1}{2}\,\varrho_0\,\|u_{,t}\|^2+\,\frac{3}{2}\,\varrho_0\,\eta\,\tau_{\rm c}^2\,\,\,\zeta_{,t}^2+
\mu_{\rm e}  \|\sym\, \mathrm{D}u\|^2-2\,
\mu_{\rm e}  \tr (\mathrm{D}u)\, \zeta+
3\,\mu_{\rm e}  \zeta^2\notag\\&+ \frac{\lambda_{\rm e}}{2}\, [\tr(\mathrm{D}u )]^2-2\,\frac{\lambda_{\rm e}}{2}\, \tr(\mathrm{D}u)\,3\, \zeta+\frac{\lambda_{\rm e}}{2}\, 9\,\zeta^2+\frac{3}{2}(2\,\mu_{\rm micro} + 3\,\lambda_{\rm micro}) \zeta^2  +\mu_{\rm e}L_{\rm c}^2{a_2}\|\mathrm{D}\, \zeta\|^2\bigg)\,dv\,
\\
=\int_\Omega\bigg(&\frac{1}{2}\,\varrho_0\,\|u_{,t}\|^2+\,\frac{3}{2}\,\varrho_0\,\eta\,\tau_{\rm c}^2\,\,\,\zeta_{,t}^2+
\mu_{\rm e}  \|\dev\,\sym\, \mathrm{D}u\|^2+ \frac{3\lambda_{\rm e}+2\mu_{\rm e}}{6}\, [\tr(\mathrm{D}u )]^2\notag\\&-(2\,
\mu_{\rm e}  +3\,\lambda_{\rm e})\, \tr(\mathrm{D}u)\, \zeta+\frac{3}{2}(2\,\mu_{\rm e} + 3\,\lambda_{\rm e}+2\,\mu_{\rm micro} + 3\,\lambda_{\rm micro}) \zeta^2 +\mu_{\rm e}L_{\rm c}^2{a_2}\|\mathrm{D}\, \zeta\|^2\bigg)\,dv\,\notag
\end{align}
and the power functional is
$
\Pi(t)=\int_\Omega (\dd\langle f,{u}_{,t}\rangle +\langle M,\zeta_{,t}\cdot \id\rangle)\, dv=\int_\Omega (\dd\langle f,{u}_{,t}\rangle +\underbrace{\tr(M)}_{:=\ell}\,\zeta _{,t}\rangle)\, dv\, .
$

A direct identification of the coefficients, by comparing the energies in the dislocation form (written with $\Curl$) with the specific internal energy and the kinetic energy in the Cowin-Nunziato form for the micro-voids model \eqref{energyvoid}, shows that the coefficient of the micro-voids theory can be expressed in terms of our constitutive coefficients as
\begin{align}\label{iden}
\varkappa&=3\,\eta \,\tau_c^2, \quad\quad \quad\quad\alpha=2\, \mu_{\rm e}L_c^2a_2,\quad\quad \quad\quad\beta =-(2\,\mu_{\rm e}+3\lambda_{\rm e})=-3\kappa_{\rm e},\\
\xi &=3(2\mu_{\rm e}+3\lambda_{\rm e}+2\mu_{\rm micro}+3\lambda_{\rm micro})=9(\kappa_{\rm e}+\kappa_{\rm micro}).\notag
\end{align}

Let us notice some differences between the micro-voids model and the micro-voids theory proposed by Cowin and Nunziato \cite{CowinNunziato83}:
\begin{itemize}
	\item in our version of the micro-voids model the constitutive parameter $\beta$ is not an independent coefficient, as it is in the micro-voids model. It depends on the elastic Lam\'e coefficients $\lambda_{\rm e}$ and $\mu_{\rm e}$.
	\item for elastic materials having a positive definite internal energy, the constitutive parameter $\beta$ is always negative. In some numerical simulation from the literature $\beta$ is taken positive in the context of the micro-voids model \cite{puri1985plane} since the physical relevance of this parameter was not properly understood. 
	\item in our version of the micro-voids model we show how the constitutive parameter $\xi$ is depending on the micro-, and macro-constitutive parameters. This gives them a physical interpretation.
	\item in our version of the micro-voids model there are no ad hoc defined quantities, e.g., $h$ and $g$ in the micro-voids model. All the quantities are defined by the force stress tensor $\sigma$, the moment stress tensor $m$ and by the microstress tensor $\sigma_{\rm micro}$.
\end{itemize}

\begin{remark}
	For the micro-voids model to be operative, we need $a_2>0$, i.e., the presence of $\|\skw \,\Curl P\|^2$ in the expression of the internal energy density.
\end{remark}
We introduce the action functional of the considered system by
\begin{align}
\mathcal{A}=\int_{0}^{T}\int_\Omega\bigg(&\frac{1}{2}\,\varrho_0\,\|u_{,t}\|^2+\,\frac{3}{2}\,\varrho_0\,\eta\,\tau_{\rm c}^2\,\zeta_{,t}^2-
\mu_{\rm e}  \|\sym (\mathrm{D}u -\zeta\!\!\cdot\!\! \id)\|^2-\frac{\lambda_{\rm e}}{2}\, [\tr(\mathrm{D}u -\zeta\!\!\cdot\!\! \id)]^2\notag\\&-\frac{3}{2}(2\,\mu_{\rm micro} + 3\,\lambda_{\rm micro}) \,\zeta^2  -\mu_{\rm e}L_{\rm c}^2{a_2}\|\mathrm{D} \zeta\|^2\bigg)\,dv+\int_{0}^{T}\int_\Omega (\dd\langle f,{u}_{,t}\rangle +\tr(M)\,\zeta _{,t}\rangle)\, dv\,dt.
\end{align}
The corresponding Euler-Lagrange equation for $u$ and $\zeta$ are
\begin{align}\label{eqisaxl}
\varrho_0\,u_{,tt}&=\Div[\underbrace{2\,\mu_{\rm e} \, \sym\,(\mathrm{D}u-\zeta\!\!\cdot\!\! \id) + \lambda_{\rm e}\, \tr(\mathrm{D}u -\zeta\!\!\cdot\!\! \id){\cdot} \id}_{=:\sigma \ \text{the symmetric force-stress tensor}}+f\, ,\\
\varrho_0\,3\,\eta\,\tau_{\rm c}^2\,\zeta_{,tt}&={\rm div}\!\!\!\!\!\!\!\!\!\!\!\!\!\!\underbrace{[2\,\mu_{\rm e}L_{\rm c}^2\,{a_2}\,{\rm D}\zeta]}_{=:h \ \text{the equilibrated micro-stress vector}} \!\!\!\!\!\!\!\!\!\!\!\!\!\!\underbrace{-3(2\mu_{\rm e}+3\lambda_{\rm e}+2\mu_{\rm micro}+3\lambda_{\rm micro})\zeta-(2\,\mu_{\rm e}+3\lambda_{\rm e}){\rm div} \,u}_{=:g \ \text{the equilibrated micro-force}}+\tr \,{M}\, ,\notag
\end{align}
in $\Omega\times [0,T]$, which, in view of our parameter identification \eqref{iden} and  setting  $\ell:=\tr(M)$, is in complete agreement with the equations proposed in the Cowin-Nunziato theory \cite{NunziatoCowin79}, i.e.,
\begin{align}
\mu_{\rm e}\sum_{l=1}^3\frac{\partial^2 u_r }{\partial x_l^2}+(\lambda_{\rm e}+\mu_{\rm e})\dd\sum_{l=1}^3\frac{\partial^2 u_l }{\partial x_l\partial x_r}+\beta \,\frac{\partial \zeta }{\partial x_r}+f&=\varrho_0 \,\frac{\partial^2 \,u_r}{\partial\, t^2},\qquad r=1,2,3,
\vspace{1.2mm}\notag\\
\alpha\sum_{l=1}^3\frac{\partial^2 \zeta }{\partial x_l^2}-\xi \zeta-\,\beta  \,\sum_{l=1}^3\frac{\partial u_l }{\partial x_l} +\ell&=\varrho_0 \, \varkappa\,\frac{\partial^2 \,\zeta}{\partial\, t^2}.\label{PDE}
\end{align}

Hence, in the case of isotropic materials we have
\begin{align} \sigma &=2\,\mu_{\rm e}\,
\varepsilon+\beta \,\zeta\, \id+\lambda_{\rm e}\, \tr (\varepsilon) \, \id=2\,\mu_{\rm e}\,
\sym \,{\rm D}u+\lambda_{\rm e}\, ({\rm div}\, u)\, \id +\beta \,\zeta\, \id,
\end{align}
while the natural Neumann-boundary conditions are
\begin{align}
\sigma.n=\widetilde{\sigma} \qquad \text{and} \qquad h.n=\widetilde{h},
\end{align}
 where  $h=\alpha\,{\rm D}\zeta$ is the only reminiscent part of the moment stress tensor, describing the internal forces inside the microstructure (the equilibrated  stress vector from the Cowin-Nunziato theory), and $\widetilde{\sigma}$ and $\widetilde{h}$ are given functions.
 
Note that the positivity conditions for the Cowin-Nunziato theory with micro-voids  are \begin{align}\label{const_poz_def}
3\lambda_{\rm e}+2\mu_{\rm e}>0,\qquad \mu_{\rm e}>0, \qquad \xi >0, \qquad \alpha>0, \qquad ({3\lambda_{\rm e}+2\mu_{\rm e}})\,\xi >3\,\beta ^{2}.
\end{align}
while in our  version of the micro-voids   model in dislocation format the positivity conditions are obvious 
\begin{align}\label{const_poz_def0} \mu_{\rm e}>0,\quad\quad\quad 2\mu_{\rm e}+3\lambda_{\rm e}>0,\quad\quad\quad2\mu_{\rm micro}+3\lambda_{\rm micro}>0,\quad\quad\quad a_2>0\,.
\end{align}
In any case, in view of the identifications the systems \eqref{iden}, \eqref{const_poz_def}  and \eqref{const_poz_def0} are equivalent.
Let us remark that 
\begin{align}\label{const_poz_def1}
3\lambda_{\rm e}+2\mu_{\rm e}>0,\qquad \mu_{\rm e}>0, 
\end{align}
are equivalent to
\begin{align}
E  >0, \qquad  -1< \nu<\frac{1}{2},
\end{align}
where $E  =\frac{\mu_{\rm e}\, (3\lambda_{\rm e}+2\mu_{\rm e})}{\lambda_{\rm e}+\mu_{\rm e}}$ and $\nu  =\frac{\lambda_{\rm e}}{2(\lambda_{\rm e}+\mu_{\rm e})}$, represent the Young's modulus and the Poisson's ratio, respectively.

In the following we assume $\varrho_0>0$ and $\varkappa>0$ without mentioning these conditions in the hypothesis of our results.

Note that in the relaxed micromophic model for $P=\zeta\!\!\cdot\!\! \id$  we have only four parameters, because  $2\mu_{\rm micro}+3\lambda_{\rm micro}$ can be regarded as a single parameter, instead of five considered by Cowin and Nunziato \cite{CowinNunziato83}. Letting  $\mu_{\rm micro}\rightarrow+\infty$ and $\kappa_{\rm micro}\rightarrow+\infty$ (or $\mu_{\rm micro}\rightarrow+\infty$ and $\lambda_{\rm micro}>\mathrm{const}$) we have that $\mu_{\rm e}\rightarrow\mu_{\rm macro}$, $\kappa_{\rm e}\rightarrow\kappa_{\rm macro}$, $\lambda_{\rm e}\rightarrow\lambda_{\rm macro}$ and that the internal energy density remains bounded only for vanishing $\zeta$. In this limit case the total energy is \begin{align}
\mathcal{L}_{\rm macro}=\int_\Omega\bigg(&\frac{1}{2}\varrho_0\|u_{,t}\|^2+\mu_{\rm macro}\,\| \sym\, {\rm D}u \|^2+ \frac{\lambda_{\rm macro}}{2}\, \tr({\rm D}u)^2\bigg)\, dv,
\end{align}
i.e., we arrive in a transparent way from the relaxed micromorphic model to the framework of  classical linear elasticity, via the model of micro-voids.

Regarding the existence of real plane waves, we will not use yet the general conditions for real planar waves (if and only if  in 
\eqref{relrw}), since due to the simplification of the system of partial differential equations some other cases may occur in the micro-voids model.

\section{Real plane waves in the model of isotropic materials with voids  }\setcounter{equation}{0}
We say that there exists  {\it real plane waves}  in the  direction $m=(m_1,m_2,m_3)$, $\lVert m\rVert^2=1$,  if for every  wave number $k>0$  the system of partial differential equations \eqref{PDE} admits a solution in the form
\begin{align}\label{ansatzwp}
u(x_1,x_2,x_3,t)&=\underbrace{\begin{footnotesize}\begin{pmatrix}\widehat{u}_1\\\widehat{u}_2\\\widehat{u}_3\end{pmatrix}\end{footnotesize}}_{=:\,\widehat{u}}
\, e^{{\rm i}\, \left(k\langle m,\, x\rangle_{\mathbb{R}^3}-\,\omega \,t\right)}\,,\\ \zeta (x_1,x_2,x_3,t)&={\rm i}\,\widehat{\zeta} e^{{\rm i}\, \left(k\langle m,\, x\rangle_{\mathbb{R}^3}-\,\omega \,t\right)},\quad 
\widehat{u}\in\mathbb{C}^{3}, \quad \widehat{ \zeta   }\in\mathbb{C}, \quad (\widehat{u}, \widehat{ \zeta   })^T\neq 0\,,\notag
\end{align}
only for  real frequencies $\omega\in \mathbb{R}$, where ${\rm i}\,=\sqrt{-1}$ is the complex unit.  The plane wave is called ``real'' since it is defined by real values of $\omega$. The speed of propagation is given by $v=\frac{\omega}{k}$. For the definition of $\zeta$ we take ${\rm i}\,\widehat{\zeta }$  since this choice will lead us in the end to deal  only with real valued matrices. 

From \eqref{PDE} and \eqref{ansatzwp}  it follows  that $\widehat{u}_1, \widehat{u}_2, \widehat{u}_3$ and $\widehat{\zeta }$ have to be the solutions of the following linear algebraic system
\begin{align}\label{algPDE}
-\omega^2\varrho_0 \, \widehat{u}_r&=-k^2\mu_{\rm e}\,\dd {\widehat{u}_r} -k^2\,(\mu_{\rm e}+\lambda_{\rm e})\dd\sum_{l=1}^3\widehat{u}_l\, m_l\,m_r -\beta \,k\,\widehat{\zeta }\,m_r,
\vspace{1.2mm}\qquad  r=1,2,3,\\
-{\rm i}\,\omega^2\varrho_0 \,\varkappa\, \widehat{ \zeta }&= -{\rm i}\,\alpha\,k^2\,\widehat{\zeta }-{\rm i}\,\xi \,\widehat{ \zeta } -{\rm i}\,\beta \,k\,\sum_{l=1}^3\widehat{u}_l \,m_l. \notag
\end{align}

Since our formulation is for isotropic materials, by  demanding real plane waves in any direction $m=(m_1,m_2,m_3)$, $\lVert m\rVert=1$, it is equivalent to demand real plane waves with $m_1=1$ and $m_2=m_3=0$. In this setting the system \eqref{algPDE} reads
\begin{align}\label{ansatzwp1}
-\omega^2\varrho_0 \, \widehat{u}_1&=-k^2\mu_{\rm e}\,\dd {\widehat{u}_1} -k^2\,(\mu_{\rm e}+\lambda_{\rm e})\dd\widehat{u}_1 -\beta \,k\,\widehat{\zeta },
\vspace{1.2mm}\notag\\
-\omega^2\varrho_0 \, \widehat{u}_2&=-k^2\mu_{\rm e}\,\dd {\widehat{u}_2}, \vspace{1.2mm}\\
-\omega^2\varrho_0 \, \widehat{u}_3&=-k^2\mu_{\rm e}\,\dd {\widehat{u}_3}, \vspace{1.2mm}\notag\\
-\omega^2\varrho_0 \,\varkappa\, \widehat{ \zeta }&= -\alpha\,k^2\,\widehat{\zeta }-\xi \,\widehat{ \zeta } -\beta \,k\,\widehat{u}_1. \notag
\end{align}   

We split the above system in the following two systems
\begin{align}\label{syst1}
\left[ {Q}_1(k)-\omega^2\widehat{\id}_2\,\right] 
\begin{footnotesize}
\begin{pmatrix} \widehat{u}_1 \\
\widehat{\zeta} \end{pmatrix}\end{footnotesize}=0 ,
\end{align} 
and 
\begin{align}\label{syst2}
\left(\frac{\mu_{\rm e}\,k^2}{\varrho_0 }-\omega^2\,\right) \widehat{u}_2=0, \qquad 
\left(\frac{\mu_{\rm e}\,k^2}{\varrho_0 }-\omega^2\,\right) \widehat{u}_3=0, \vspace{1.2mm}
\end{align}   
where
\begin{align}
 { {Q}_1}(k)&=\begin{footnotesize}\begin{pmatrix}
(\lambda_{\rm e}+2\, \mu_{\rm e} )k^2& \beta \,k\vspace{2mm}\\
\beta \,k & \alpha\,k^2+\xi 
\end{pmatrix}\end{footnotesize},\qquad
\widehat{\id}_2=\begin{footnotesize}\begin{pmatrix} \varrho_0  & 0 \\
0 & \varrho_0 \,\varkappa\, \end{pmatrix}\end{footnotesize}.
\end{align}

In order to have an eigenvalue problem, we multiply \eqref{syst1} from the left with $\widehat{\id}_2^{-1/2}$ and make a substitution to obtain
\begin{align}\label{syst1_2}
\left[\widehat{\id}_2^{-1/2} { {Q}_1}(k)\widehat{\id}_2^{-1/2}-\omega^2\id\right]\,  {d}=0,\qquad\qquad   {d}=\widehat{\id}_2^{1/2}\begin{footnotesize}\begin{pmatrix}
\widehat{u}_1 \\ \widehat{ \zeta }
\end{pmatrix}\end{footnotesize}.
\end{align}
Hence, the system \eqref{syst1} is equivalent to the eigenvalue problem
\begin{align}\label{syst1_3}
\left[\widetilde{ {Q}}(k)-\omega^2\,\id\,\right] {d} 
=0 ,\qquad\qquad \text{where}\qquad \qquad  
\widetilde{ {Q}}(k)=\widehat{\id}_2^{-1/2} { {Q}_1}(k)\widehat{\id}_2^{-1/2}\in\Sym(2).
\end{align} 

The system \eqref{syst2} admits a non-zero solution if and only if $\omega^2=k^2\,\frac{\mu_{\rm e}}{\varrho_0 }$, while the system \eqref{syst1} (i.e., \eqref{syst1_3}) has a non-zero solution if and only if $\det{\left[\widetilde{ {Q}}(k)-\omega^2\,\id\,\right]}=0$. Having real waves implies that
\begin{align}
\mu_{\rm e}>0
\end{align} and that $\omega^{2}$ are positive real eigenvalues of $\widetilde{Q}$, i.e.,  
\begin{align}
\widetilde{ { {Q}}}(k)&=\begin{footnotesize}\begin{pmatrix}
\frac{(\lambda_{\rm e}+2\, \mu_{\rm e} )k^2}{\varrho_0 } & \frac{\beta \,k}{\varrho_0 \,\sqrt{\varkappa}}\vspace{2mm}\\
\frac{\beta \,k}{\varrho_0 \,\sqrt{\varkappa}} & \frac{\alpha\,k^2+\xi }{\varrho_0 \,\varkappa}
\end{pmatrix}\end{footnotesize}
\end{align}
must be positive definite.
But $\widetilde{ {Q}}(k)$ is  positive definite if and only if  \begin{align}\label{rwc}\lambda_{\rm e}+2\mu_{\rm e}>0, \qquad \alpha>0,\qquad  (\lambda_{\rm e}+2\mu_{\rm e})\,\xi >\beta ^2.\end{align}

In conclusion, we have obtained:
\begin{proposition} There exists  {\it real plane waves}  in the linear model of isotropic elastic materials   with micro-voids     if and only if  
\begin{align}\label{rwc1}\mu_{\rm e}>0, \qquad \lambda_{\rm e}+2\mu_{\rm e}>0, \qquad \alpha>0,\qquad  (\lambda_{\rm e}+2\mu_{\rm e})\xi >\beta ^2.\end{align} 
\end{proposition}
In terms of the constitutive parameters of the relaxed micromorphic model the conditions \eqref{rwc1} read
\begin{align}\label{rwc2}
\mu_{\rm e}>0, \qquad \lambda_{\rm e}+2\mu_{\rm e}>0, \qquad a_2>0,\qquad  (\lambda_{\rm e}+2\mu_{\rm e})(\kappa_{\rm e}+\kappa_{\rm micro})>\kappa_{\rm e}^2.
\end{align}
As we observe, comparing to the analogue conditions for the parental relaxed micromorphic model, it follows that 
\begin{align}
\kappa_{\rm e}+\kappa_{\rm micro}>0.
\end{align}
In the micro-voids framework the real plane waves existence do not imply that $\mu_{\rm micro}>$ or $\mu_{\rm e}+\mu_{\rm micro}>0$. 

The conditions \eqref{rwc1} are equivalent with the \emph{strong ellipticity} \cite{ChiritaGhiba1} 
of the elastic material with micro-voids, i.e., with the condition
\begin{equation}
{W}\left(  m\otimes\eta,\theta,{\vartheta}\right)  >0,\text{ \ \ for all
}\ \ \ m,\eta,\vartheta\in\mathbb{R}^3 \setminus\{{0}\}, \theta\in \mathbb{R},
\label{3.1}%
\end{equation}
where
\begin{align}
{W}\left(  X,c,b\right)=\mu_{\rm e}\,\| \sym X \|^2+ \frac{\lambda_{\rm e}}{2}\, \tr(X)^2+\frac{\xi }{2}\,c^2+\frac{\alpha}{2}\,\|b\|^2+b \,\tr(X)\,c.
\end{align}

Moreover, for an isotropic elastic material with micro-voids, the progressive wave (\ref{ansatzwp}) can
propagate in the direction $m=(m_1,m_2,m_3)\neq 0$ with the speeds $v_{1}=v_{2}$, $v_{3}$ and $v_{4}$, see \cite{ChiritaGhiba1}, where%
\begin{align}\label{4.8}
v_{1}^{2}&=v_{2}^{2}=v_t^2:=\frac{\mu_{\rm e}}{{ \varrho_0}}, %
\notag\\\dd
v_{3}^{2}&=\frac{1}{2{ \varrho_0}\varkappa}\left\{
\xi+\alpha+\varkappa\left( \lambda_{\rm e}+2\mu_{\rm e}\right) +\sqrt{\left[
	\xi+\alpha-\varkappa \left( \lambda_{\rm e}+2\mu_{\rm e}\right)  \right]
	^{2}+4\,\varkappa{\beta }
	^{2}}\right\}, \vspace{1mm}\\
\dd v_{4}^{2}&=\frac{1}{2{ \varrho_0}\varkappa}\left\{
\xi+\alpha+\varkappa\left( \lambda_{\rm e}+2\mu_{\rm e}\right) -\sqrt{\left[
	\xi+\alpha-\varkappa \left( \lambda_{\rm e}+2\mu_{\rm e}\right)  \right]
	^{2}+4\,\varkappa{
\beta }\,
	^{2}}\right\} . \notag\end{align}

 The progressive
waves propagating with speeds $v=v_{1}=v_{2}=\sqrt
{\displaystyle\frac{\mu_{\rm e}}{{ \varrho_0}}} $ are transverse waves and  the characteristic space (the solutions of \eqref{algPDE}) is generated by the
linear independent vectors 
$
\widehat{u}^{\left(  1\right)  }=\left(
-m_{3},0,m_{1}\right)  
$  \text{and} $\qquad  \widehat{u}^{\left(  2\right)
}=\left(  -m_{2},m_{1},0\right) .
$
For $v=v_{3}$ and $v=v_4$ the progressive waves are longitudinal waves with the corresponding solution  of the algebraic system
\eqref{algPDE}
given by 
$
\widehat{u}^{(3)} =\left(
	m_{1}\mathfrak{c},m_{2}\mathfrak{c},m_{3}\mathfrak{c},{\beta } \right) , $
\text{and}
$
\widehat{u}^{(4)} =\left(
	m_{1}\mathfrak{a},m_{2}\mathfrak{a},m_{3}\mathfrak{a},{\beta } \right), $
respectively, where
$
\mathfrak{c}={ \varrho_0}\varkappa v_{3}^{2}-\xi -\alpha$  \text{and} $
\mathfrak{a}={ \varrho_0}\varkappa v_{4}^{2}-\xi -\alpha. %
$

In the next sections we show that there exists a unique Rayleigh wave solution for all isotropic elastic materials   with micro-voids    satisfying \eqref{rwc1}, i.e., in the same range of values for the constitutive coefficients for which real progressive plane waves exist.

\section{Rayleigh waves in isotropic elastic materials with micro-voids }\label{setRw}\setcounter{equation}{0}

In this section we consider that the material   with micro-voids     occupying
the half-space 
 $$\Sigma:=\{(x_1,x_2,x_3)\,|\,x_1,x_3\in \mathbb{R},\,\, x_2\geq0\}$$  is homogeneous and isotropic and its  boundary  is free of surface traction, i.e.,
\begin{align}\label{03}
{\sigma}.\, n=\sigma_{i3}=0,\qquad \qquad {h}.\, n=h_2=0 \qquad \qquad\textrm{for}\quad x_2=0.
\end{align}

 Without loss of generality we will study the waves propagating
 along the $x_{1}-$axis. 
 In the context of Rayleigh wave propagation, the surface particles move in the planes normal to the surface $x_2=0$ and parallel to the direction of propagation. In consequence, we consider the following  plane strain ansatz as a first step in our process of construction of the solution
 \begin{align}
u= u(x_1,x_2,t)&=\begin{footnotesize}\begin{pmatrix} u_1(x_1,x_2,t)
 \\
 u_2(x_1,x_2,t)
 \\0
 \end{pmatrix}\end{footnotesize},\qquad \qquad  \zeta=\zeta(x_1,x_2,t).\label{ansatz1}
 \end{align}

 Beside the boundary condition \eqref{03} we require that the solutions be attenuated in the direction $x_{2}%
 $, so that they are decaying with distance from the plane surface
 $x_{2}=0$,
 that is we require that
 \begin{align}\label{04} \lim_{x_2 \rightarrow \infty}\{u_1,u_2,{\zeta},\sigma_{12},\sigma_{21},\sigma_{22},h_2\} (x_1,x_2,t)=0 \qquad \quad \forall\, x_1\in\mathbb{R}, \quad \forall \, t\in[0,\infty).
 \end{align}

The further aim of this paper is to give an explicit solution  $(u_1,u_2,\zeta)   $ of the system of partial differential equations which results from \eqref{PDE} by considering the plane strain ansatz \eqref{ansatz1}, i.e., a solution of 
 \begin{align}\label{x6}
 \varrho_0 \, \frac{\partial^2{u}_{1}}{ \partial \,t^2}&=(\lambda_{\rm e}+2\,\mu_{\rm e})\,\frac{\partial^2 u_{1}}{\partial \,x_1^2}+(\lambda_{\rm e}+\mu_{\rm e}) \,  \frac{\partial^2 u_{2}}{\partial \,x_2\partial\, x_1}+ \mu_{\rm e}\,\frac{\partial^2 u_{1}}{\partial \,x_2^2}+\beta \, \frac{\partial\,\zeta}{\partial \,x_1},\notag\\
\varrho_0 \, \frac{\partial^2{u}_{2}}{\partial \,t^2}&=\mu_{\rm e}\,\frac{\partial^2 u_{2}}{\partial \,x_1^2}+(\lambda_{\rm e}+\mu_{\rm e})\,\frac{\partial^2 u_{1}}{\partial \,x_2\,\partial\,x_1}+(\lambda_{\rm e}+2\mu_{\rm e})\,\frac{\partial^2 u_{2}}{\partial \,x_2^2}+\beta \, \frac{\partial\,\zeta}{\partial \,x_2},\\
\varrho_0 \,\varkappa\,\frac{\partial^2 {{\zeta}}}{\partial \,t^2}&=\alpha\,\frac{\partial^2\zeta}{\partial\,x_1^2}+\alpha\,\frac{\partial^2\zeta}{\partial\,x_2^2}-\xi \zeta-\beta \,\frac{\partial\, u_{1}}{\partial \,x_1}-\beta \,\frac{\partial\, u_{1}}{\partial \,x_2},\notag
 \end{align}
 which satisfies the boundary conditions at $x_2=0$
 \begin{align}\label{x17}
 \mu_{\rm e}\,\frac{\partial\,u_{1}}{\partial \,x_2}+\mu_{\rm e}\,\frac{\partial\,u_{2}}{\partial \,x_1}=\,0,\qquad
 (\lambda_{\rm e}+2\mu_{\rm e}) \,\frac{\partial\,u_{2}}{\partial \,x_2}+\lambda_{\rm e} \,  \frac{\partial\,u_{1}}{\partial \,x_1}+\beta \zeta=\,0,\qquad 
 \alpha \,\frac{\partial\,\zeta}{\partial \,x_2}=\,0,
 \end{align}
and has the asymptotic behaviour \eqref{04}.

\section{The algebraic analysis  of a Riccati-type equation}\setcounter{equation}{0}

	Let 
	$  {\mathcal{X}},   {\mathcal{Y}}\in \mathbb{R}^{3\times 3}$ be positive definite  matrices and $  {\mathcal{Z}}\in \mathbb{R}^{3\times 3}$. 
	
	\begin{definition}\label{deflim}By the {\bf limiting speed} associated to 	$  {\mathcal{X}},   {\mathcal{Y}}\in \mathbb{R}^{3\times 3}$  and $  {\mathcal{Z}}\in \mathbb{R}^{3\times 3}$ we understand a speed $\widehat{v}>0$, such that for all wave speeds satisfying 	$0\le v<\widehat{v}$  (subsonic speeds) the solutions $r$ of  \begin{align}\label{x90}
	\det\,[r^2  {\mathcal{X}}+r(  {\mathcal{Z}}+  {\mathcal{Z}}^T)+  {\mathcal{Y}}-k^2 v^2 {\id}]=0,
	\end{align}
	are not real. 
\end{definition}
In other words, if the roots  of the characteristic equation \eqref{x90} are not real then they correspond to wave speeds $v$ satisfying 	$0\le v<\widehat{v}$.

In this paper, we will use a matrix algebraic analysis  of a Riccati-type equation, see \eqref{120},  presented by Fu and Mielke in \cite{fu2001nonlinear,fu2002new} (see also the work by Mielke  and Sprenger \cite{mielke1998quasiconvexity}), i.e.,
\begin{theorem}\label{fmth} {\rm (Fu and Mielke in \cite{fu2001nonlinear,fu2002new})}
	Let 
	$  {\mathcal{X}},   {\mathcal{Y}}\in \mathbb{R}^{3\times 3}$ be positive definite  matrices and $  {\mathcal{Z}}\in \mathbb{R}^{3\times 3}$. 
	\begin{enumerate}
		\item 	If $0\leq v<\widehat{v}$, then the matrix problem 
		\begin{align}\label{13}
		  {\mathcal{X}}{\mathcal{E}}^2- {\rm i}\, (  {\mathcal{Z}}+   {\mathcal{Z}}^T)\mathcal{E}-   {\mathcal{Y}}+ k^2 v^2{\id}=0,\qquad \textrm{\rm Re\,spec\,}  {\mathcal{E}}>0,
		\end{align}
		where $\textrm{\rm Re\,spec\,}  {\mathcal{E}} $  means the real part of spectra of $  {\mathcal{E}}$  has a unique solution for $  {\mathcal{E}}$ and $  {  {\mathcal{M}}}:=-(-	  {\mathcal{X}}  {\mathcal{E}}+ {\rm i}\,  {\mathcal{Z}}^T)$ is Hermitian.

		\item 
		If $0\leq v<\widehat{v}$, then the unique solution $\mathcal{M}$ of the \textit{algebraic Riccati equation} 
		\begin{align}\label{120}
		(  {\mathcal{M}}- {\rm i}\,  {\mathcal{Z}})  {\mathcal{X}}^{-1}(  {\mathcal{M}}+ {\rm i}\, \mathcal{R^T})-  {\mathcal{Y}}+k\, v^2\,{\id}=0, \qquad  {  {\mathcal{M}}}\, y(0)=0.
		\end{align}
		that satisfies  $\text{\rm Re\,spec}\,(   {\mathcal{X}}^{-1}(  {\mathcal{M}}+{\rm i}\,  {\mathcal{Z}}^T))>0$ is given explicitly by
		\begin{align}\label{fmv}
		  {\mathcal{M}}=\Big(\int_{0}^{\pi}	   {\mathcal{X}}_  { \theta} ^{-1}\, d\theta\Big)^{-1}\Big(\pi\id-  {\rm i}\, \int_{0}^{\pi}   {\mathcal{X}}_  { \theta} ^{-1}   {   {\mathcal{Z}}}_\theta  ^T\, d\theta\Big),
		\end{align}
		where we write  $\widetilde{  {\mathcal{Y}}}=  {\mathcal{Y}}-k\,v^2\,{\id}$,  $\theta$ is an arbitrary angle, while  the  matrices $  {\mathcal{X}}_  { \theta}$, $  {\mathcal{Y}}_  { \theta}$ and $  {\mathcal{Z}}_  { \theta}$ are  obtained by rotation of the old coordinate system\footnote{We remark that   $  {\mathcal{X}}$ and $  {\mathcal{Y}}$ remain symmetric and  that
			$\widetilde{  {\mathcal{Y}}}_\theta$, $  {\mathcal{X}}_  { \theta} $ and $  {\mathcal{Z}}_\theta  $ are periodic in $\theta$ with periodicity $\pi$ and
			\begin{align}\label{RTQ}
			\widetilde{  {\mathcal{Y}}}_\theta(\theta+\frac{\pi}{2})=  {\mathcal{X}}_  { \theta} ,\qquad \quad  {\mathcal{Z}}_\theta(\theta+\frac{\pi}{2})=-  {\mathcal{Z}}_\theta  ^T,\qquad\quad    {\mathcal{X}}_  { \theta}(\theta+\frac{\pi}{2})=\widetilde{  {\mathcal{Y}}}_\theta.
			\end{align}
			In addition, according to the definition of the limiting speed, regarding Proposition \ref{lemmaGH}  and  Proposition  \ref{propQp},
			the limiting velocity $\widehat{v}$ is in fact the lowest velocity for which the matrices $\widetilde{  {\mathcal{Y}}}_\theta$ and $T(\theta)$  become singular for some angle $\theta$ and  $\widetilde{  {\mathcal{Y}}}_\theta$ is positive definite for $0\leq v<\widehat{v}$, see the proof of Proposition \ref{lemmaGH} for more details. Hence, from \eqref{RTQ} we have that  $  {\mathcal{X}}_  { \theta} $ is positive definite, too. 
			Thus by the definition of the limiting speed $\widehat{v}$ both $\widetilde{  {\mathcal{Y}}}_\theta$ and $  {\mathcal{X}}_  { \theta} $ are positive definite or positive semi-definite depending on  $\theta$ (for $v=\widehat{v}$ there is at least one $\theta$ at which $  {\mathcal{X}}(\theta)$ has an eigenvalue 0, and likewise $\widetilde{  {\mathcal{Y}}}_\theta$).} about $e_3$ by an angle $\theta$
		\begin{align}\label{21}
		  {\mathcal{X}}_  { \theta} =\cos^2\theta\,  {\mathcal{X}}-\sin\theta\cos\theta\,(  {\mathcal{Z}}+  {\mathcal{Z}}^T)+\sin^2\theta\,\widetilde{  {\mathcal{Y}}},\notag\vspace{2mm}\\
		  {\mathcal{Z}}_\theta  =\cos^2\theta\,  {\mathcal{Z}}+\sin\theta\cos\theta
		\,(  {\mathcal{X}}-\widetilde{  {\mathcal{Y}}})-\sin^2\theta\,  {\mathcal{Z}}^T,\vspace{2mm}\\
		\widetilde{  {\mathcal{Y}}}_\theta=\cos^2\theta\,\widetilde{  {\mathcal{Y}}}+\sin\theta\cos\theta\,(  {\mathcal{Z}}+  {\mathcal{Z}}^T)+\sin^2\theta\,  {\mathcal{X}}.\notag
		\end{align}
		\item 	If $0\leq v<\widehat{v}$, then  the solution  $  {\mathcal{M}}_v$ obtained from \eqref{120} has the following properties
		\begin{enumerate}
			\item $  {\mathcal{M}}_v$ is Hermitian,
			\item  $\frac{d  {  {\mathcal{M}}}_v}{d v}$ is negative definite,
			\item  $\tr(  {  {\mathcal{M}}}_v)\geq 0$, and $\langle  w,  {  {\mathcal{M}}}_v\,w\rangle \geq 0$ for all real vectors $w$ for  all $0\leq v\leq  \widehat{v}$,
			\item $  {\mathcal{M}}_v$ is and positive definite for  all $0\leq v< \widehat{v}$.
		\end{enumerate}
		\item The secular equation
		\begin{align}
		\det  {  {\mathcal{M}}}_v=0,
		\end{align}
		where $  {\mathcal{M}}_v$ is obtained from $\eqref{120}$, has a unique admissible solution $0\leq v<\widehat{v}$.
	\end{enumerate}
\end{theorem}

\section{The ansatz for the solution and the limiting speed}\label{anzsec}\setcounter{equation}{0}

\subsection{The ansatz}

We look for a solution of \eqref{x6} and \eqref{x17} having the form\footnote{We take ${\rm i}\,z_3$ since this choice leads us, in the end, only to deal with real matrices.}
 \begin{align}\label{x5}
 \mathcal{U}(x_1,x_2,t)=\begin{footnotesize}\begin{pmatrix}u_1(x_1,x_2,t)
 	\\
 	u_2(x_1,x_2,t)
 	\\ \zeta(x_1,x_2,t)
 	\end{pmatrix}\end{footnotesize}={\rm Re}\left[\begin{footnotesize}\begin{pmatrix} z_1(x_2)
 	\\
 	z_2(x_2)
 	\\{\rm i}\,z_3(x_2)
 \end{pmatrix}\end{footnotesize} e^{ {\rm i}\, k\, (  x_1-vt)}\right],
 \end{align}
 where  $v$ is the propagation speed {(the phase  velocity)}. If the Cauchy problem given by the following system
 \begin{align}&
\begin{footnotesize}\begin{pmatrix} 
\mu_{\rm e} &0	&0
 	\\
 	0& \lambda_{\rm e}+2\,\mu_{\rm e} &0
 	\\
0 & 0 & \alpha \end{pmatrix}\end{footnotesize}\begin{footnotesize}\begin{pmatrix}	z_1''(x_2)
 	\\
 	z_2''(x_2)
 	\\\	z_3''(x_2)
 	\end{pmatrix}\end{footnotesize}+ {\rm i}\, \begin{footnotesize}\begin{pmatrix} 0& k\,(\lambda_{\rm e}+\mu_{\rm e} )& 0  
 	\\
 	k\,(\mu_{\rm e} +\lambda_{\rm e})& 0 & \beta 
 	\\
0 & \beta  & 0 \end{pmatrix}\end{footnotesize}\begin{footnotesize}\begin{pmatrix} 	z_1'(x_2)
 	\\
 	z_2'(x_2)
 	\\\	z_3'(x_2)
 	\end{pmatrix}\end{footnotesize}\notag\vspace{2mm}\\&\qquad\qquad \ \ -\begin{footnotesize}\begin{pmatrix} 
k^2\,(2\,\mu_{\rm e}+\lambda_{\rm e})-\varrho_0  \,k^2v^2 & 0	& \beta \,k
 	\\
0 & k^2\,\mu_{\rm e}-\varrho_0  \,k^2v^2 &0
 	\\
\beta \,k & 0 & k^2\,\alpha+ \xi  -\varrho_0 \,\varkappa\, k^2v^2 \end{pmatrix}\end{footnotesize}\begin{footnotesize}\begin{pmatrix} 	z_1(x_2)
 	\\
 	z_2(x_2)
 	\\	z_3(x_2)
 	\end{pmatrix}\end{footnotesize}=0,
 \end{align}
 and the boundary conditions
 \begin{align}
\begin{footnotesize}\begin{pmatrix} \mu_{\rm e} &0	&0
 	\\
 	0& 2\,\mu_{\rm e} +\lambda_{\rm e} & 0
 	\\
 	0 & 0 & \alpha \end{pmatrix}\end{footnotesize}\begin{footnotesize}\begin{pmatrix} 	z_1'(0)
 	\\
 	z_2'(0)
 	\\	z_3'(0)
 	\end{pmatrix}\end{footnotesize}+ {\rm i}\, \begin{footnotesize}\begin{pmatrix} 			0 & k\,\mu_{\rm e} & 0  
 	\\
 	k\,\lambda_{\rm e} & 0 & \beta 
 	\\
 	0 & 0 & 0  \end{pmatrix}\end{footnotesize}\begin{footnotesize}\begin{pmatrix}	z_1(0)
 	\\
 	z_2(0)
 	\\	z_3(0)
 \end{pmatrix}\end{footnotesize}=0,
 \end{align}\normalsize
has the solution $z_i$, $i=1,2,3$, where $\cdot '$ denotes the derivative with respect to $x_2$, then $\mathcal{U}$ given by the ansatz \eqref{x5} satisfies \eqref{x6} and \eqref{x17}. In  a more compact notation, the above equations reads
\begin{align}\label{11}
\frac{1}{k^2}\,{ {T}}\,z''(x_2)+ {\rm i}\,\frac{1}{k}\, ({ {R}}+{ {R}}^T)\,z'(x_2)-{ {Q}}\,z(x_2)+k^2\, v^2 \, \hat\id\,z(x_2)=\,0,\qquad 
\frac{1}{k^2}\,{ {T}}\,z'(0)+ {\rm i}\, \frac{1}{k}\,{ {R}}^T\,z(0)=&\,0,
\end{align} 
where the matrices  ${ {T}}\,\,,{ {R}}, \, {Q}$ and $\widehat{\id}^{-\frac{1}{2}}$  are defined by
\begin{align}\label{x47}
 { {T}}=k^2\begin{footnotesize}\begin{pmatrix}  \mu_{\rm e} &0	&0
 	\\
 	0& 2\,\mu_{\rm e}+\lambda_{\rm e} & 0
 	\\
 	0 & 0 & \alpha \end{pmatrix}\end{footnotesize},
 	\qquad  { {R}}=k\begin{footnotesize}\begin{pmatrix}
   0& k\lambda_{\rm e} & 0
 	\\
 	k\,\mu_{\rm e} & 0 & 0
 	\\
 	0 & \beta  & 0 \end{pmatrix}\end{footnotesize},\hspace{1.5cm}\\
 { {Q}}=\begin{footnotesize}\begin{pmatrix} 
 k^2\,(2\,\mu_{\rm e}  +\lambda_{\rm e}) & 0 & \beta \,k
 	\\
 	0 & k^2\,\mu_{\rm e} & 0
 	\\
 	\beta \,k & 0 & k^2\,\alpha+\xi  \end{pmatrix}\end{footnotesize},
 	\qquad
\widehat{\id}^{-1/2}=\begin{footnotesize}\begin{pmatrix} 
\frac{1}{\sqrt{\varrho_0 }} & 0 & 0 \\
0 & \frac{1}{\sqrt{\varrho_0 }} & 0 \\
0 & 0 & \frac{1}{\sqrt{\varrho_0 \,\varkappa}}\end{pmatrix}\end{footnotesize}.\notag
 \end{align}
To find  the  solution $z(x_2)$ of \eqref{11} is equivalent to finding a  solution $y(x_2):=\widehat{\id}^{1/2}\,z(x_2)$ of 
\begin{align}\label{11t}
 \frac{1}{k^2}\,\widehat{\id}^{-1/2}\,{ {T}}\,\widehat{\id}^{-1/2}\,y''(x_2)+ {\rm i}\,\frac{1}{k}\, \widehat{\id}^{-1/2}\,({ {R}}+{ {R}}^T)\widehat{\id}^{-1/2}\,y'(x_2) -\widehat{\id}^{-1/2}\,{ {Q}}\,\widehat{\id}^{-1/2}\,y(x_2)+k^2\, v^2 \, \id\,y(x_2)&=\,0,\\
  \frac{1}{k^2}\,\widehat{\id}^{-1/2}\,{ {T}}\,\widehat{\id}^{-1/2}\,y'(0)+ {\rm i}\, \frac{1}{k}\,\widehat{\id}^{-1/2}\,{ {R}}^T\,\widehat{\id}^{-1/2}\,y(0)&=\,0.\notag
 \end{align}

With the help of the modified matrices
 \begin{align}\label{nTQ}
   {\mathcal{T}}:=\widehat{\id}^{-1/2}\,{ {T}}\,\widehat{\id}^{-1/2}:=\,k^2\begin{footnotesize}\begin{pmatrix} 
 \frac{\mu_{\rm e}}{\varrho_0 } &0	&0
 \\
 0 & \frac{2\,\mu_{\rm e} +\lambda_{\rm e}}{\varrho_0 }  & 0
 \\
 0  & 0 & \frac{\alpha}{\varrho_0 \,\varkappa} \end{pmatrix}\end{footnotesize}, \qquad 
  {\mathcal{R}}:=\widehat{\id}^{-1/2}\,{ {R}}\,\widehat{\id}^{-1/2}:= k\,\begin{footnotesize}\begin{pmatrix} 
0 & k\,\frac{\lambda_{\rm e} }{\varrho_0 }& 0
 \\
 k\,\frac{\mu_{\rm e}}{\varrho_0 }&0  &0
 \\
 0  & \frac{\beta }{\varrho_0 \sqrt{\varkappa}} & 0 \end{pmatrix}\end{footnotesize},
 \\
  {\mathcal{Q}}:=\widehat{\id}^{-1/2}\,{ {Q}}\,\widehat{\id}^{-1/2}:=\begin{footnotesize}\begin{pmatrix} 
\frac{2\,\mu_{\rm e}+\lambda_{\rm e}}{\varrho_0 }k^2 & 0	&\frac{\beta \,k}{\varrho_0 \sqrt{\varkappa}}
 \\
 0 & \frac{\mu_{\rm e}}{\varrho_0 }k^2 & 0
 \\
 \frac{\beta \,k}{\varrho_0 \sqrt{\varkappa}} & 0 & \,\frac{\alpha\,k^2+\xi }{\varrho_0 \,\varkappa}
  \end{pmatrix}\end{footnotesize}\hspace{3.5cm}\notag
 \end{align}
 the system \eqref{11} turns into
 \begin{align}\label{n11}
 \frac{1}{k^2}  {\mathcal{T}}\,y''(x_2)+ {\rm i}\,\frac{1}{k} (  {\mathcal{R}}+  {\mathcal{R}}^T)\,y'(x_2)-  {\mathcal{Q}}\,y(x_2)+k^2\, v^2 \, \id\,y(x_2)=\,0,\qquad
 \frac{1}{k^2}  {\mathcal{T}}\,y'(0)+ {\rm i}\,\frac{1}{k} \,  {\mathcal{R}}^T\,y(0)=&\,0.
 \end{align}
 
 \begin{lemma}\label{lpd} If  the constitutive coefficients satisfy the conditions 	$\mu_{\rm e}>0,\ \ \lambda_{\rm e}+2\mu_{\rm e}>0,\ \ \alpha>0$ and \linebreak $(\lambda_{\rm e}+2\mu_{\rm e})\,\xi >\beta ^2$
 then the	 matrices $  {\mathcal{Q}}$ and $  {\mathcal{T}}$ are symmetric and positive definite.	
 \end{lemma}
 \begin{proof}
 	The symmetry is clear.  
 	The matrix $ {T}$ is positive definite if and only if $\mu_{\rm e}>0,\,\lambda_{\rm e}+2\mu_{\rm e}>0,\,\alpha>0$ and so the matrix $  {\mathcal{T}}$ is positive definite, being a product of positive definite matrices.
 	
 	In addition,  ${ {Q}}$ is positive-definite if and only if the principal minors are positive, namely
 	\begin{align}
 	{ {Q}}_{11}&=k^2(\lambda_{\rm e}+2\,\mu_{\rm e}),\qquad
 	{ {Q}}_{11}{ {Q}}_{22}-{ {Q}}_{12}{ {Q}}_{21}=k^4\,\mu_{\rm e}(\lambda_{\rm e}+ 2\,\mu_{\rm e})\,,\\
 	\det({ {Q}})&=k^4\,\mu_{\rm e}\left[(\lambda_{\rm e}+ 2\,\mu_{\rm e})\alpha\,k^2+(\lambda_{\rm e}+ 2\,\mu_{\rm e})\xi-\beta ^2\right]\,,\notag
 	\end{align}
 i.e., under the hypothesis of the lemma. 
 Since $ {Q}$ is positive definite, so it is $  {\mathcal{Q}}$ defined by \eqref{nTQ}, and the proof is complete.
 \end{proof} 
 
 \subsection{The existence of the limiting speed in the model}

We now  seek  a solution $y$ of the differential system \eqref{n11} in the form 
 \begin{align}\label{x7}
 y(x_2)=\begin{footnotesize}\begin{pmatrix} 	d_1
 	\\
 	d_2
 	\\d_3
 	\end{pmatrix}\end{footnotesize} \,e^{{\rm i}\,r\,k\,x_2},\qquad \text{Im}\,r>0,
 \end{align}
 where $r \, \in \mathbb{C}$ is a complex parameter, $d=\begin{footnotesize}\begin{pmatrix} 	d_1,
 &
 d_2,
 &d_3
 \end{pmatrix}\end{footnotesize}^T \in \mathbb{C}^3 $, $d
 \neq 0$ is the amplitude and $ \text{Im}\,r$ is the coefficient of the imaginary part of $r$. 
 
 From \eqref{x7} and $\eqref{n11}$ we obtain the  system
 \begin{align}\label{x8}
 [r^2  {\mathcal{T}}+r\,(  {\mathcal{R}}+  {\mathcal{R}}^T)+  {\mathcal{Q}}-\ k^2 v^2 {\id}]\,d=0,\qquad \qquad 
 [r\,  {\mathcal{T}}+  {\mathcal{R}}^T]\,d=0.
 \end{align}
 
Because \eqref{x8}$_1$ is a characteristic equation corresponding to an eigenvalue problem, the eigenvectors $d=\begin{footnotesize}\begin{pmatrix} 	d_1,
&
 d_2,
&d_3
 \end{pmatrix}\end{footnotesize}^T\neq0$  if 
  \begin{align}\label{x9}
 \det\,[r^2  {\mathcal{T}}+r(  {\mathcal{R}}+  {\mathcal{R}}^T)+  {\mathcal{Q}}-k^2 v^2 {\id}]=0,
 \end{align}
 which is an equation of order 6 with the unknown  $r$. We will search the conditions for having solutions for \eqref{x8} with $ \text{Im}\,r>0$, which ensures the asymptotic decay condition \eqref{04}.

The {\bf limiting speed} in the micro-voids model is the limiting speed associated to 	$  {\mathcal{T}},   {\mathcal{Q}}\in \mathbb{R}^{3\times 3}$  and $  {\mathcal{R}}\in \mathbb{R}^{3\times 3}$, see Definition \ref{deflim} and Theorem \ref{fmth}.

 \begin{proposition}\label{lemmaGH} If  the constitutive coefficients satisfy the conditions 	$\mu_{\rm e}>0,\,\lambda_{\rm e}+2\mu_{\rm e}>0,\,\alpha>0$ and $(\lambda_{\rm e}+2\mu_{\rm e})\,\xi >\beta ^2$,
 	then there	exists a limiting speed $\widehat{v}>0$ and moreover if one root   $r_v$ of the characteristic equation \eqref{x9} is real then it corresponds to a speed $v\geq \widehat{v}$ (non-admissible).
 \end{proposition}
 \begin{proof}
Assume that there exists a real $r_v$ as solution  of the characteristic equation \eqref{x9}, then $ \exists \,\theta \in (-\frac{\pi}{2},\frac{\pi}{2})$ such that $r_v=\tan \theta$. 
Since
\begin{align}
e^{ik\, (  x_1+r x_2-vt)}=e^{ik\, (  x_1+\tan\theta x_2-vt)}=e^{\frac{ik}{\cos\theta}(\cos\theta x_1+\sin\theta x_2-\cos\theta vt)},
\end{align}
and since the solution $\mathcal{U}$  given by \eqref{x5} can be written using \eqref{x7} in the form
 \begin{align}\label{46}
 \mathcal{U}(x_1,x_2,t)&=\begin{footnotesize}
\begin{pmatrix} u_1(x_1,x_2,t)
 	\\
 	u_2(x_1,x_2,t)
 	\\\zeta(x_1,x_2,t)
 	\end{pmatrix}\end{footnotesize}
 	={\rm Re}\left[\begin{footnotesize}
 		\begin{pmatrix} a_1
 	\\
 	a_2
 	\\ {\rm i}\, \,a_3
 	\end{pmatrix}\end{footnotesize}e^{\frac{ik}{\cos\theta}(\cos\theta x_1+\sin\theta x_2-\cos\theta vt)}\right],
 \end{align}
where $(a_1,a_2,a_3)^T=\widehat{\id}^{-1/2}(d_1,d_2,d_3)^T$,  we are looking to $\mathcal{U}(x_1,x_2,t)$ as to    a non-trivial  plane body wave solution with  wave number $k_{\theta}=\frac{k}{\cos\theta}$, the speed $v_{\theta}=v\cos\theta$, the same frequencies $\omega_\theta=k_{\theta}v_{\theta}=kv=\omega$  and propagation in the direction ${n}_\theta$ where 
 $	{n}_\theta=(\cos\theta,\sin\theta,0)$. 
 
 This is the reason why, in the following, we examine the conditions under which we observe "real" bulk waves only for positive real values of $\omega^2$ and the presence of a non-trivial solution of \eqref{PDE}  having  the form
 \begin{align}\label{solutionRay}
 u(x_1,x_2,t)&=\begin{footnotesize}\begin{pmatrix} u_1(x_1,x_2,t)
 \\
 u_2(x_1,x_2,t)
 \\0
 \end{pmatrix}\end{footnotesize}
 =\begin{footnotesize}\begin{pmatrix}\widehat{u}_1
 \\
 \widehat{u}_2
 \\
 0
 \end{pmatrix}\end{footnotesize}\, e^{{\rm i}\, \left(k\langle m,\, x\rangle_{\mathbb{R}^3}-\,\omega \,t\right)},\qquad
 \zeta(x_1,x_2,t)
 =i\,\widehat{\zeta}\, e^{{\rm i}\, \left(k\langle m,\, x\rangle_{\mathbb{R}^3}-\,\omega \,t\right)},
 \\&\qquad \qquad 
 \qquad \qquad 
 (\widehat{u}_1,\widehat{u}_2, \widehat{ \zeta   })^T\in\mathbb{C}^{3}, \quad (\widehat{u}_1,\widehat{u}_2, \widehat{ \zeta   })^T\neq 0\,\notag
 \end{align}
 for a wave propagation direction  $m = (m_1,m_2,0)^T$ with $\lVert{m}\rVert^2=1$ and for every  wave number $k>0$.  With this conditions, system \eqref{PDE} becomes
 \begin{align}\label{PDE_Ray}
 \left[ {Q}_{2}(k)-\omega^2\widehat{\id}\,\right] 
 \begin{footnotesize}
 \begin{pmatrix} 
 \widehat{u}_1 \\
 \widehat{u}_2 \\
 \widehat{\zeta} 
 \end{pmatrix}\end{footnotesize}
 =\begin{footnotesize}
 \begin{pmatrix} 
 0 \\
 0 \\
 0 
 \end{pmatrix}\end{footnotesize} ,
 \end{align} 
 where
 \begin{align}\label{xx1}
 { {Q}_2}(k)&=
 \begin{footnotesize}\begin{pmatrix}
 (\lambda_{\rm e}+2\mu_{\rm e} )k^2m_1^2+\mu_{\rm e}\,k^2m_2^2 & (\lambda_{\rm e}+\mu_{\rm e} )k^2m_1m_2 & \beta \,k\,m a_1 \vspace{2mm}\\
 (\lambda_{\rm e}+\mu_{\rm e} )k^2m_1m_2 & (\lambda_{\rm e}+2\mu_{\rm e} )k^2m_2^2+\mu_{\rm e}\,k^2m_1^2 & \beta \,k\,m_2 \vspace{2mm}\\
 \beta \,k\,m_1 & \beta \,k\,m_2 & \alpha\,k^2+\xi 
 \end{pmatrix}\end{footnotesize},\qquad
 \widehat{\id}=\begin{footnotesize}\begin{pmatrix} 
 \varrho_0  & 0 & 0 \\
 0 & \varrho_0  & 0 \\
 0 & 0 & \varrho_0 \,\varkappa\, \end{pmatrix}\end{footnotesize}.
 \end{align}
 An equivalent form of \eqref{xx1} is
 \begin{align}\label{PDE_Ray_tilda}
 \left[ {\widetilde{Q}_{2}}(k)-\omega^2\,\id\,\right] 
 \begin{footnotesize}
 \begin{pmatrix} 
 \widetilde{u}_1 \\
 \widetilde{u}_2 \\
 \widetilde{\zeta} 
 \end{pmatrix}\end{footnotesize}
 =\begin{footnotesize}
 \begin{pmatrix} 
 0 \\
 0 \\
 0 
 \end{pmatrix}\end{footnotesize} , \qquad 
 (\widetilde{u}_1, \widetilde{u}_2, \widetilde{\zeta})^T=\widehat{\id}^{1/2}(\widehat{u}_1, \widehat{u}_2, \widehat{\zeta} )^T.
 \end{align} 
 
 The matrix $ { {Q}_2}(k)$  is positive definite for all $k>0$ if and only if
 \begin{align}
m_1^2 (\lambda_{\rm e} +2 \mu_{\rm e} )+m_2^2 \mu_{\rm e}>0, \quad  \mu_{\rm e}  (\lambda_{\rm e} +2 \mu_{\rm e} )>0, 
\qquad k^2\left[\alpha   \mu_{\rm e}  (\lambda_{\rm e} +2 \mu_{\rm e} )\right]+\left[ \mu_{\rm e}  \left(\xi   (\lambda_{\rm e} +2 \mu_{\rm e} )-\beta  ^2\right)\right]>0,
 \end{align}
 for all $k>0$. It is easy to remark that 
 the above inequalities are satisfied if and only if $\mu_{\rm e}>0,\,\lambda_{\rm e}+2\mu_{\rm e}>0,\,\alpha>0$ and $(\lambda_{\rm e}+2\mu_{\rm e})\,\xi >\beta ^2$.
 But $ { {Q}_2}(k)$ is positive definite if and only if $ {\widetilde{ {Q}}_2}(k)=\widehat{\id}^{-1/2} { {Q}_2}(k)\,\widehat{\id}^{-1/2}$ is positive definite, which  further is equivalent to the fact that  the equation
$
 \det\left[ {\widetilde{Q}_{2}}(k)-\omega^2\,\id\,\right] =0
 $ 
 has only real positive solutions for $\omega^2$.
 Hence, we have proven that for every     direction of propagation of the form $m=(m_1,m_2, 0)^T$, $\lVert m\rVert^2=1$, the necessary and sufficient conditions for existence of  a non trivial solution of the form  \eqref{solutionRay} for the system of partial differential equations \eqref{PDE} are  $\mu_{\rm e}>0,\,\lambda_{\rm e}+2\mu_{\rm e}>0,\,\alpha>0$ and $(\lambda_{\rm e}+2\mu_{\rm e})\,\xi >\beta ^2$

By knowing this, we may proceed further to our analysis.  A direct substitution of \eqref{46} into \eqref{algPDE} and
 using the same idea as in the case of obtaining the formulas \eqref{11t}-\eqref{n11} we have that there  exists   a non-trivial solution  $\begin{footnotesize}\begin{pmatrix}d_1,&d_2,&d_3\end{pmatrix}\end{footnotesize}\neq 0$ of the algebraic system written  in  matrix form 
\begin{align}\label{x093}
\left[	\sin^2\theta\,  {\mathcal{T}}+\sin\theta\cos\theta (  {\mathcal{R}}+  {\mathcal{R}}^T)+\cos^2\theta\,  {\mathcal{Q}}-k^2 v^2\cos^2\theta \,{\id}\right]\begin{footnotesize}\begin{pmatrix} 	d_1
\\
d_2
\\	d_3
\end{pmatrix}\end{footnotesize}=0,
\end{align}
with notations from \eqref{nTQ}. Let us remark that the equation \eqref{x093} is the propagation condition for plane waves in isotropic materials with micro-voids, in the direction $	{n}_\theta=(\cos\theta,\sin\theta,0)$. 
According to the information of the previous paragraph, for the direction $	{n}_\theta=(\cos\theta,\sin\theta,0)$ in particular, the system of partial differential equations \eqref{PDE} admits a non trivial solution in the form given by \eqref{solutionRay} only for real positive values $\omega^2$.
Therefore, 
to each  $\theta \in(-\frac{\pi}{2},\frac{\pi}{2})$  we can associate the  real frequencies $\omega_\theta$ satisfying 
\begin{align}\label{omega_theta}
\det\,\{	\sin^2\theta  {\mathcal{T}}+\sin\theta\cos\theta (  {\mathcal{R}}+  {\mathcal{R}}^T)+\cos^2\theta  {\mathcal{Q}}-\omega^2_\theta \,{\id}\}=0,
\end{align}
and further  a propagation speed $v_\theta=v\cos\theta$ such that $ \omega_\theta=k\,v_\theta$ that verify the equation
\begin{align}\label{x11}
 \det\,\{	\sin^2\theta  {\mathcal{T}}+\sin\theta\cos\theta (  {\mathcal{R}}+  {\mathcal{R}}^T)+\cos^2\theta  {\mathcal{Q}}-k^2 v^2_\theta  \,{\id}\}=0.
 \end{align}
 
 We define the limiting speed $\widehat{v}$ as the minimum of the values of $v_\theta=v\,\cos\theta$ for all $\theta\in(-\frac{\pi}{2},\frac{\pi}{2})$
 \begin{align}
 \widehat{v}=\inf_{\theta\in(-\frac{\pi}{2},\frac{\pi}{2})} v_\theta.\label{lims}
 \end{align}
 
 In conclusion, for a real $r_v$ there exists a $ \theta \in (-\frac{\pi}{2},\frac{\pi}{2})$ such that $r_v=\tan \theta$. For this $\theta$ we have the  real frequencies $\omega_\theta$ satisfying \eqref{omega_theta} and the propagation speed $v_\theta$ corresponding to $r_v$ is  satisfying \eqref{x11} with $v_\theta\geq\widehat{v}$.
Thus,  if $v$ is such that
 $
 0\le v<\widehat{v},
 $
 then $r_v$ can not be real and if $r_v$ is real, then $v\geq\widehat{v}$.
\end{proof}

\begin{proposition} If  the constitutive coefficients satisfy the conditions 	$\mu_{\rm e}>0,\,\lambda_{\rm e}+2\mu_{\rm e}>0,\,\alpha>0$ and $(\lambda_{\rm e}+2\mu_{\rm e})\,\xi >\beta ^2$, then
	for all $\theta\in \left(-\frac{\pi}{2},\frac{\pi}{2}\right)$ and $k>0$, the tensor $  {\mathcal{Q}}_\theta:=\sin^2\theta  {\mathcal{T}}+\sin\theta\cos\theta (  {\mathcal{R}}+  {\mathcal{R}}^T)+\cos^2\theta  {\mathcal{Q}}$ is positive definite.
\end{proposition}
\begin{proof} The proof is straightforward since 
from the equation \eqref{x093} we have that the tensor $  {\mathcal{Q}}_\theta$ satisfies
\begin{align}\label{x093_1}
\left[	  {\mathcal{Q}}_\theta-k^2 v^2\cos^2\theta \,{\id}\right]\begin{footnotesize}\begin{pmatrix} 	d_1
\\
d_2
\\	d_3
\end{pmatrix}\end{footnotesize}=0,
\end{align}
which means that if we take the particular propagation direction $m=(\cos\theta,\sin\theta,0)$ we obtain that $ {\widetilde{Q}_{2}}(k)=  {\mathcal{Q}}_\theta$ is positive definite. 
\end{proof}

\begin{proposition}\label{propQp}
If  the constitutive coefficients satisfy the conditions $\mu_{\rm e}>0,\,\lambda_{\rm e}+2\mu_{\rm e}>0,\,\alpha>0$ and $(\lambda_{\rm e}+2\mu_{\rm e})\,\xi >\beta ^2$, then
for all $\theta\in \left(-\frac{\pi}{2},\frac{\pi}{2}\right)$, $k>0$ and $0\le v<\widehat{v}$, the tensor $\widetilde{ {Q}}_\theta:=  {\mathcal{Q}}_\theta-k^2 v^2 \cos^2\theta \,{\id}$ is positive definite.
\end{proposition}
\begin{proof}
We have to prove that    all eigenvalues $  {\mathcal{Q}}_\theta$ are larger than those of the matrix $k^2 v^2 \cos^2\theta \,{\id}$, i.e., than $k^2 v^2 \cos^2\theta $.   Note that since $  {\mathcal{Q}}_\theta:=\sin^2\theta  {\mathcal{T}}+\sin\theta\cos\theta (  {\mathcal{R}}+  {\mathcal{R}}^T)+\cos^2\theta  {\mathcal{Q}}$ is positive definite, it admits only positive eigenvalues. Assuming that there exist  $\theta_0\in \left(-\frac{\pi}{2},\frac{\pi}{2}\right)$ and  $v_0\in[0,\widehat{v})$ for which there is an eigenvalue $\Lambda_{\theta_0}$   of $  {\mathcal{Q}}_\theta$
such that $\Lambda_{\theta_0}<k^2 v^2_0 \cos^2\theta_0 $, then 
\begin{align}
v_{\theta_0}:=\sqrt{\frac{\Lambda_{\theta_0}}{k^2\,\cos^2\theta_0}}<v_0<\widehat{v}
\end{align}
is solution of \eqref{x11}, i.e., for fixed $\theta_0$ we have that $v_{\theta_0}<\widehat{v}$  verifies
 \begin{align}\label{x11n}
\det\,\{	\sin^2{\theta_0}\,  {\mathcal{T}}+\sin{\theta_0}\,\cos{\theta_0} \,(  {\mathcal{R}}+  {\mathcal{R}}^T)+\cos^2{\theta_0}\,  {\mathcal{Q}}-k^2 v^2_{{\theta_0}}\, \cos^2{\theta_0} \,{\id}\}=0.
\end{align} 
This  contradicts  the definition of the limiting speed and Proposition \ref{lims}, since $\widehat{v}$ is the smallest speed having this property. Therefore, it remains that for all  $\theta\in \left(-\frac{\pi}{2},\frac{\pi}{2}\right)$, $k>0$ and for all $0\le v<\widehat{v}$, all the eigenvalues of $  {\mathcal{Q}}_\theta$ are larger than $k^2 v^2 \cos^2\theta$ and the proof is complete.
\end{proof}

Gathering together the calculations done for the construction of the Rayleigh wave solution given by Chirit\u a and Ghiba \cite{ChiritaGhiba2}, we have the explicit form of the solutions of \eqref{x9}  given as
\begin{align}
r_1^{2}=k^{2}\left(  1-\frac{c^{2}}{v_{t}^{2}}\right),\quad \qquad
r_{2,3}^{2}  =  k^{2}\left(  1-\displaystyle\frac{T_{\mp}}{k^{2}%
}\right)\qquad \text{where}\quad  v_{t}=\sqrt{\frac{\mu_{\rm e}}{{ \varrho_0}}}
\end{align}
 and
\begin{align}
T_{\mp}(\omega^{2})  &
=\displaystyle\frac{1}{2\alpha(\lambda_{\rm e}+2\mu_{\rm e}
	)}\Bigg\{{\beta }^{2}-(\xi -{ \varrho_0}\varkappa\omega^{2})(\lambda_{\rm e}+2\mu_{\rm e})+\alpha
\omega^{2}{ \varrho_0}\nonumber\\
&
\qquad \qquad \qquad \quad \mp\sqrt{4\,\alpha\,\omega^{2}{\beta }^{2}{ \varrho_0}+[-{\beta }^{2}+\alpha\omega
	^{2}{ \varrho_0}+(\xi -{ \varrho_0}\varkappa\omega^{2})(\lambda_{\rm e}+2\mu_{\rm e})]^{2}}\Bigg\}.
\end{align}
In view of the definition of the limiting speed, from Theorem \ref{fmth} we conclude:
\begin{proposition}\label{lemmaCG1} If  the constitutive coefficients satisfy the conditions $\mu_{\rm e}>0,\,\lambda_{\rm e}+2\mu_{\rm e}>0,\,\alpha>0$ and $(\lambda_{\rm e}+2\mu_{\rm e})\,\xi >\beta ^2$
	then 	 the roots  $r_v$  of the characteristic equation \eqref{x9} are not real if and only if  
	\begin{equation}
	v^{2}<\min\{v_{s}^{2},{v}_{m}^{2}\}. \label{condspeed3}%
	\end{equation}
	where
	\begin{align}\label{v36} {v}_{s}^2&:={\displaystyle\frac{2\,\mu_{\rm e}  +\lambda_{\rm e}}{ \varrho_0}}, \notag\\v_{m}^2&: = \displaystyle\frac{1}{2k^{2}{ \varrho_0}\varkappa
	}\Big\{\xi +k^{2}(\lambda_{\rm e}+2\mu_{\rm e})\varkappa+k^{2}\alpha\vspace{1mm}\\
	& \qquad \qquad\ \  -\sqrt{(\xi +k^{2}(\lambda_{\rm e}+2\mu_{\rm e})\varkappa+k^{2}\alpha)^{2}-4k^{2}%
		\varkappa\lbrack\xi (\lambda_{\rm e}+2\mu_{\rm e})-{\beta }^{2}+k^{2}\alpha(\lambda_{\rm e}+2\mu_{\rm e}
		)]}\Big\}.
	\notag 
	\end{align}
	This defines the limiting speed in isotropic elastic materials   with micro-voids    to  be
	\begin{align}\label{limCh}
	\widehat{v}:=\inf_{\theta\in(-\frac{\pi}{2},\frac{\pi}{2})} v_\theta\equiv \min\{v_{s}^{2},{v}_{m}^{2}\}.
	\end{align}
\end{proposition}

\section{Existence and uniqueness of Rayleigh waves }\label{NSE}\setcounter{equation}{0}

 Following the ideea by Fu and Mielke \cite{fu2002new} we are now looking at \eqref{n11} as  to an initial value problem and we search the solution in the form 
\begin{align}\label{08}
y(x_2)=e^{-k\,x_2 \,  {\mathcal{E}}}y(0),
\end{align}
where $  {\mathcal{E}}\in\mathbb{C}^{3 \times 3}$  is to be determined as solution (see  \eqref{08} and  \eqref{x8}) of
\begin{align}\label{09}
[  {\mathcal{T}}  {\mathcal{E}}^2- {\rm i}\, (  {\mathcal{R}}+  {\mathcal{R}}^T)  {\mathcal{E}}-  {\mathcal{Q}}+\ k^2v^2 {\id}]\, y(x_2)=0, \qquad\qquad \qquad  (-  {\mathcal{T}}  {\mathcal{E}}+ {\rm i}\, \,  {\mathcal{R}}^T)\,y(0)=0.
\end{align}
In order  to have a proper decay, the eigenvalues of $  {\mathcal{E}}$ have to be such that their real part is positive. The unknowns in \eqref{09} are $  {\mathcal{E}}\in\mathbb{C}^{3 \times 3}$ and $v$.

We are doing a change of variable by introducing  the so called  \textit{surface impedance matrix} 
\begin{align}\label{10}
  {  {\mathcal{M}}}=-(-	  {\mathcal{T}}  {\mathcal{E}}+ {\rm i}\,  {\mathcal{R}}^T)\qquad \iff \qquad   {\mathcal{E}}=  {\mathcal{T}}^{-1}(  {\mathcal{M}}+ {\rm i}\, \,  {\mathcal{R}}^T),
\end{align}
which, by   substituting $\eqref{10}_2$ into $\eqref{09}_1$, must be  a solution of the \textit{algebraic Riccati equation} 
\begin{align}\label{12}
(  {\mathcal{M}}- {\rm i}\,  {\mathcal{R}})  {\mathcal{T}}^{-1}(  {\mathcal{M}}+ {\rm i}\, \mathcal{R^T})-  {\mathcal{Q}}+k\, v^2\,{\id}=0, \qquad\qquad   {  {\mathcal{M}}}\, y(0)=0.
\end{align}

Since we are interested in a nontrivial solution  $y$, we impose $y(0)
\neq0$, so that the matrix $  {\mathcal{M}}$ has to satisfy
\begin{align}\label{x16}
{\rm det}\,  {\mathcal{M}}=0.
\end{align}
The  equation \eqref{x16} is called  \textit{secular equation} for the micro-voids model in terms of the \textit{impedance matrix} $  {\mathcal{M}}$.

The strategy to solve the system of nonlinear equations \eqref{12}$_1$ and \eqref{x16} is the following. We find an admissible domain for $v$ such that, for fixed $v$ in this domain we may define the mapping $v\mapsto  {  {\mathcal{M}}}_v$, where $  {\mathcal{M}}_v$ is the solution of \eqref{12}$_1$. Note that not any $  {\mathcal{M}}_v$ is admissible because $  {\mathcal{M}}_v$ has to be such that $\text{Re(spec}\,  {\mathcal{E}}{\rm )}$ is positive, 	where ``$\text{\rm Re\,(spec\,}  {\mathcal{E}}{\rm )}$" means the ``real part of spectra of $  {\mathcal{E}}$"\!\!. In the end, having $  {\mathcal{M}}_v$ we find $v$ as solution of \eqref{x16}. An important aspect is that we have to be sure that the solution $v$ belongs to the admissible domain considered for having a suitable $  {\mathcal{M}}_v$ and that the solution is unique, since otherwise the uniqueness of the matrix $  {\mathcal{M}}_v$ is violated.

Since $  {\mathcal{T}}$ and $  {\mathcal{Q}}$ are symmetric and positive definite matrices, many aspects from the above discussions are purely mathematical questions, and they are not specific to the elastic materials with micro-voids.
 Indeed,  since $  {\mathcal{T}}$,  $  {\mathcal{Q}}$ and $\widehat{\id}$ are symmetric real matrices,  the  results established in \cite{mielke2004uniqueness,fu2002new} remains valid in the  framework of the  elastic materials with micro-voids. 
 
 First, in view of Lemma \eqref{lpd} and the item i) of Theorem \ref{fmth}, if  the constitutive coefficients satisfy the conditions $\mu_{\rm e}>0,\,\lambda_{\rm e}+2\mu_{\rm e}>0,\,\alpha>0$ and $(\lambda_{\rm e}+2\mu_{\rm e})\,\xi >\beta ^2$
	and   $0\leq v<\widehat{v}$, the matrix problem 
	\begin{align}\label{13E}
	  {\mathcal{T}}  {\mathcal{E}}^2- {\rm i}\, \,(  {\mathcal{R}}+  {\mathcal{R}}^T)  {\mathcal{E}}-  {\mathcal{Q}}+ k^2 v^2{\id}=0,\qquad \text{\rm Re\,(spec\,}  {\mathcal{E}}{\rm )}>0,
	\end{align}
	has a unique solution for $  {\mathcal{E}}_v$  and the corresponding matrix $  {\mathcal{M}}_v$ obtained from $\eqref{10}_2$ is Hermitian.
Hence,  we know that for all $0\leq v<\widehat{v}$ there exists a unique solution $  {\mathcal{M}}_v$ of the Riccati equation \eqref{12}, defined by  the unique matrix $  {\mathcal{E}}_v$ indicated  in the item i) of Theorem \ref{fmth}, i.e.,  we can consider the mapping which associates to each  $0\leq v<\widehat{v}$ the Hermitian matrix $  {\mathcal{M}}_v$ satisfying the equation \eqref{12}. In addition, from item ii) of Theorem \ref{fmth} we know the full representation of the unique admissible solution $  {\mathcal{M}}_v$ of the algebraic Riccati equation \eqref{12} to be given by 
\begin{align}\label{explMielke}
{\mathcal{M}}_v= {H}_v^{-1}+ {\rm i}\, \,  {H}_v^{-1}\,  {S}_v, \qquad \text{with} \qquad  {H}_v=\frac{1}{\pi}\dd\int_{0}^{\pi}	  {\mathcal{T}}_  { \theta} ^{-1}\, d\theta, \qquad   {S}_v=-\frac{1}{\pi}\int_{0}^{\pi}  {\mathcal{T}}_  { \theta} ^{-1}  {\mathcal{R}}_\theta  ^T\, d\theta,
\end{align}
where

\begin{align}\label{21}
{\mathcal{T}}_{ \theta} =\cos^2\theta\,  {\mathcal{T}}-\sin\theta\cos\theta\,(  {\mathcal{R}}+  {\mathcal{R}}^T)+\sin^2\theta\,\widetilde{  {\mathcal{Q}}},\notag\vspace{2mm}\\
{\mathcal{R}}_\theta  =\cos^2\theta\,  {\mathcal{R}}+\sin\theta\cos\theta
\,(  {\mathcal{T}}-\widetilde{  {\mathcal{Q}}})-\sin^2\theta\,  {\mathcal{R}}^T,\vspace{2mm}\\
\widetilde{  {\mathcal{Q}}}_\theta=\cos^2\theta\,\widetilde{  {\mathcal{Q}}}+\sin\theta\cos\theta\,(  {\mathcal{R}}+  {\mathcal{R}}^T)+\sin^2\theta\,  {\mathcal{T}}\notag
\end{align}
and  $\widetilde{  {\mathcal{Q}}}=  {\mathcal{Q}}-k\,v^2\,{\id}$, while $\theta$ is an arbitrary angle.
Let us point out that \eqref{explMielke} allows us to obtain the explicit   form of the secular equation $\det   {\mathcal{M}}_v=0$ without a priori knowing the analytical expressions (as function of the wave speed) of the eigenvalues  that satisfy \eqref{x9}  and the associated eigenvector, which is the major difficulty in almost all the generalised models when other methods are used.  After having the admissible solution of the secular equation, we will go back to the task of finding the eigenvalues  that satisfy \eqref{x9}  and the associated eigenvector. But, by doing so this task  becomes a purely numerical task, avoiding symbolic (analytical) computations, since all the involved quantities will be known, as numbers.

Then, after replacing  the solution $  {\mathcal{M}}_v$ of \eqref{12} into  \eqref{x16n},  the pair $(v,  {  {\mathcal{M}}}_v)$ must  also be a solution of the secular equation  \eqref{x16}, i.e., 
\begin{align}\label{x16n}
{\rm det}\,  {\mathcal{M}}_v=0.
\end{align}
The unique unknown is now $v$ and we have to check if   the resulting equation will lead to a unique wave speed belonging to the interval $[0,\widehat{v})$. This is necessary because under this assumption we have constructed the mapping $v\mapsto  {  {\mathcal{M}}}_v$ with  ${  {\mathcal{M}}}_v$ admissible. But, according to item iii) of  Theorem \ref{fmth}, the matrix $  {\mathcal{M}}_v$ determined by \eqref{explMielke} satisfies the conditions  \begin{enumerate}
	\item $  {\mathcal{M}}_v$ is Hermitian,
	\item  $\frac{d  {  {\mathcal{M}}}_v}{d v}$ is negative definite,
	\item  $\tr(  {  {\mathcal{M}}}_v)\geq 0$, and $\langle  w,  {  {\mathcal{M}}}_v\,w\rangle \geq 0$ for all real vectors $w$ for  all $0\leq v\leq  \widehat{v}$,
	\item $  {\mathcal{M}}_v$ is and positive definite for  all $0\leq v< \widehat{v}$.
\end{enumerate} and all these  imply the existence of a unique subsonic solution of the secular equation. 
Summarizing, using Theorem \ref{fmth}, we have
\begin{theorem}\label{emr}
	Assume  the constitutive coefficients satisfy the conditions 	$\mu_{\rm e}>0,\,\lambda_{\rm e}+2\mu_{\rm e}>0,\,\alpha>0$ and $(\lambda_{\rm e}+2\mu_{\rm e})\,\xi >\beta ^2$ 
	and $  {\mathcal{M}}_v$ is given by  $\eqref{explMielke}$, then the secular equation
	$
	\det  {  {\mathcal{M}}}_v=0
	$
	 has a unique admissible solution $0\leq v<\widehat{v}$.
\end{theorem}

The Theorem \ref{emr} serves as the final checkpoint in our algorithm, essentially validating the feasibility of our strategy for solving the nonlinear equations \eqref{12}$_1$ and \eqref{x16}. It conclusively demonstrates that there is indeed an unique Rayleigh wave propagating within isotropic elastic materials containing micro-voids, and that this wave is unique.
\section{Numerical implementation}\label{Num1}\setcounter{equation}{0}

In this section we follow the theoretical solution  strategy given in the previous sections to give the numerical solution for the material design properties for structural steel S235, for the following values of the Poisson ratio and the Young modulus for the  macroscopic structure
\begin{align}\label{numcoef1}
\nu_{\rm macro}=0.3,\qquad E_{\rm macro}=210 \,\textrm{(GPa)},\qquad \varrho_0=7850 \,\textrm{($kg/m^3$)},
\end{align}
and we take the wave number $k=1 \,\textrm{($m^{-1}$)}.$
Accordingly, the considered Lam\'e coefficient of the macroscopic structure are given by
\begin{align}
\lambda _{\text{macro}}&=\frac{E_{\rm macro}}{2 (\nu_{\rm macro} +1)}=121.154 \,\textrm{(GPa)},\qquad \qquad \mu _{\text{macro}}=\frac{\nu_{\rm macro}  E_{\rm macro}}{(\nu_{\rm macro} +1) (1-2\, \nu_{\rm macro} )}=80.7692 \,\textrm{(GPa)}, \\ \kappa_{\text{macro}}&=\frac{1}{3} \left(3\, \lambda _{\text{macro}}+2\, \mu _{\text{macro}}\right)=161.538\,\textrm{(GPa)}.\notag
\end{align}
We assume
\begin{align}
\lambda _{\text{micro}}=:2 \,\lambda _{\text{macro}}=242.308 \,\textrm{(GPa)}, \qquad \mu _{\text{micro}}=:2 \,\mu _{\text{macro}}=161.538 \,\textrm{(GPa)},
\end{align}
and we find the following values of the elastic constitutive coefficients 
\begin{align}
\mu _e&=-\frac{\mu _{\text{macro}}\, \mu _{\text{micro}}}{\mu _{\text{macro}}-\mu _{\text{micro}}}=242.308\,\textrm{(GPa)}, \qquad \kappa_e=-\frac{\kappa_{\text{macro}}\, \kappa_{\text{micro}}}{\kappa\kappa_{\text{macro}}-\kappa_{\text{micro}}}=323.077\,\textrm{(GPa)}, \\ \lambda _e&=\frac{1}{3} \left(3\, \kappa_e-2\, \mu _e\right)=161.538\,\textrm{(GPa)}, \notag
\end{align}
and
\begin{align}
\beta =-3\, \kappa_e=-969.231\,\textrm{(GPa)},\qquad \xi =9 \left(\kappa_e+\kappa_{\text{micro}}\right)=5815.38\,\textrm{(GPa)}.
\end{align}

We choose the values of $\eta$, $L_{\rm c}$, $a_2$ and $\tau_{\rm c}$ such that
\begin{align}
\varkappa&=3\,\eta \,\tau_c^2=1.5 \,\textrm{($m^2$)}, \quad\quad \quad\quad\alpha=2\, \mu_{\rm e}L_c^2a_2=10\,\textrm{(GPa  $m^2$)}.
\end{align}

For the chosen  material, using (\ref{limCh}) the value of the limiting speed is found to be $\widehat{v}=0.242876  \,(m/s)$. An admissible  solution of the secular equation  is a positive wave speed less then $\widehat{v}$. As we have analytically proven such a solution exists and it is unique.
We find the unique solution of the secular equation using the interpolating process on a set of 100 values in the interval $[0,\widehat{v})$. We compute numerically the integrals in (\ref{explMielke}) since we do not have the symbolic values of them. Using interpolation we find a function $v\mapsto f(v)$ that approximates the function  $v\mapsto \det\mathcal{M}_v$, on $[0,\widehat{v})$, see Figure \ref{fig:method}. The root of the approximation function $f$ on $[0,\widehat{v})$ is $v_0=0.153005 \,(m/s)$. Our numerical interpolation algorithm of the function $v\mapsto \det{{\mathcal{M}}}_v$ is made on 100 points just to exemplify the method. One can increase the number of points and use more precision options in order to increase the precision of the results. 

\begin{figure}[h!]
	\centering
	\includegraphics[width=12cm]{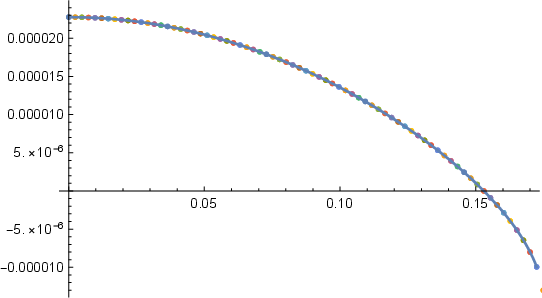}
	\caption{\footnotesize A plot of $\det\mathcal{M}_v$ with respect to the surface waves speed $v$ for a set of 100 equidistant values in the interval $[0,\widehat{v})$. It is illustrated that $\det\mathcal{M}_v$ is  a decreasing function of the wave speed $v$.}
	\label{fig:method}
\end{figure} 

Having the wave speed $v_0$, the approximation of the solution of the secular equation we find $y(0)$ as solution of the algebraic Riccati equation (\ref{12}) in the form
\begin{align}
y(0)=\left(
\begin{array}{c}
y_1 \\
-1.263037839859903454 {\rm i} \,y_1 \\
0.20925703711388041\, y_1 \\
\end{array}
\right), \qquad 	 y_1\in \mathbb{C}.
\end{align}

We also have computed  $\mathcal{M}_{v_0}$ to be
\begin{align}
\mathcal{M}_{v_0}=
\left(
\begin{array}{ccc}
0.030152425181168009480 & -0.023503969977247586063 {\rm i} & -0.0022270300650370918782 \\
0.023503969977247580765\,{\rm i} & 0.018790427909898263031 & 0.0010945940135284871257\, {\rm i} \\
-0.0022270300650370918782 & -0.0010945940135284891285 \, {\rm i} & 0.017249300922842580276 
\end{array}
\right)
\end{align}
(its determinant $\det \mathcal{M}_{v_0}=-4.135543508749075155*10^{-13}$ is approximately  zero as we expected from  Theorem \ref{emr}) and the matrix ${\mathcal{E}_{v_0}}=  {\mathcal{T}}^{-1}(  {\mathcal{M}_{v_0}}+ {\rm i}\, \,  {\mathcal{R}}^T)$ is numerically approximated by
\begin{align}
{\mathcal{E}_{v_0}}=
\left(
\begin{array}{ccc}
0.976843 & 0.238546 \,{\rm i} & -0.0721487 \\
0.535545 \,{\rm i} & 0.228281 &  -1.21145 \,{\rm i} \\
-2.62233 &  -1.28888 \,{\rm i} & 20.3111 \\
\end{array}
\right).
\end{align}

Having the matrix ${\mathcal{E}_{v_0}}$, using (\ref{08}) we are able to find the function $y(x_2)$. In this way, going back further with the changes of variables $y(x_2):=\widehat{\id}^{1/2}\,z(x_2)$ we find $z(x_2)$ and one step back, using (\ref{x5}), we find the solution $\mathcal{U}(x_1,x_2,t)$, see Figure \ref{u1u2}.
\begin{figure}[h!]
		\centering 
		\begin{subfigure}{.31\textwidth}
			\includegraphics[width=\linewidth]{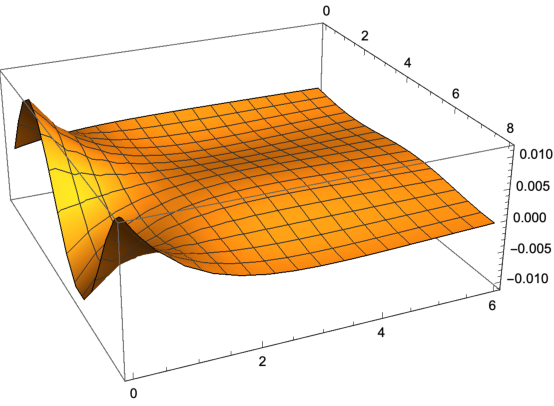}
			\caption{Plot of the $u_1$-component of the displacement.}
		\end{subfigure}\quad\  \qquad \qquad 
		\begin{subfigure}{.31\textwidth}
			\includegraphics[width=\linewidth]{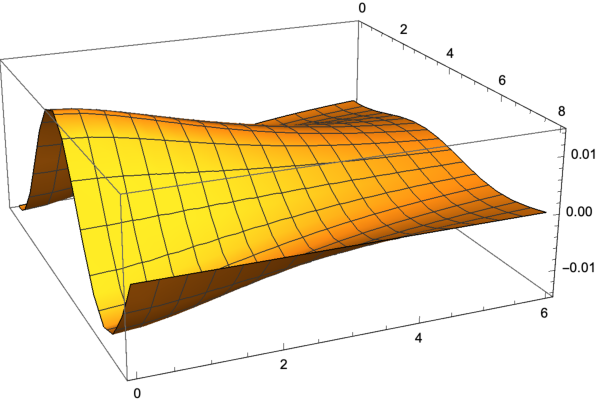}
			\caption{Plot of the $u_2$-component of the displacement.}
		\end{subfigure}\qquad \qquad \ 
		\begin{subfigure}{.35\textwidth}
			\includegraphics[width=\linewidth]{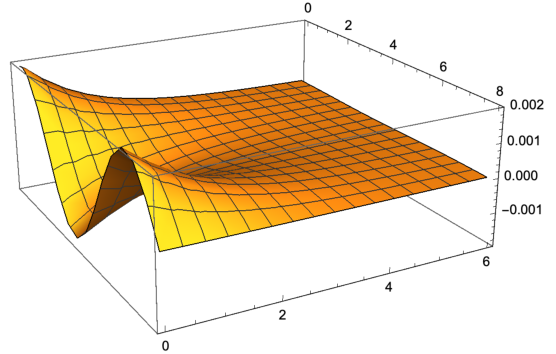}
			\caption{Plot of the volume fraction field $\zeta$ from $P=\zeta\!\!\cdot\!\! \id$.}
		\end{subfigure}
		\caption{The plot of the solution at time $t=1$ and for the choice $y_1=-{\rm i}\,$.}\label{u1u2}
	\end{figure}

\begin{figure}[h!]
	\centering
	\includegraphics[width=10cm]{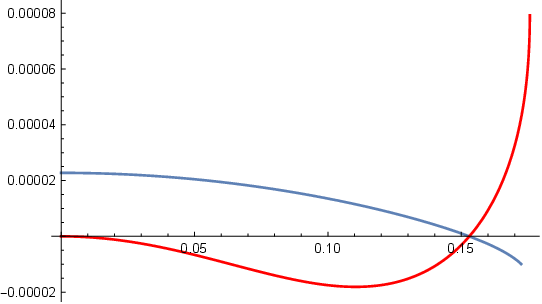}
	\put(-165,0){\footnotesize $R(c)$ function from \cite{ChiritaGhiba2}}
	\put(-150,70){\footnotesize our approach}	\caption{\footnotesize The plots of the approximation of $v\mapsto \det{{\mathcal{M}}}_v$ (blue curve) and of $v\mapsto R(v) 10^{-8}$ (red curve), where $R(v)$ from (\ref{gensec}) defines the secular equations from \cite{ChiritaGhiba2}, with respect to the surface waves speed $v$ for the  material we considered. These two different functions have the same (unique) root in the interval $[0,\widehat{v})$. We have multiplied $R(v)$ with $10^{-8}$ to have the same range plot interval of the values for both functions.}
	\label{method2}
\end{figure} 

In \cite{ChiritaGhiba2} the secular equation has the form $R(c)=0$, where $c$ is the wave speed and
\begin{align}%
R(v)\equiv&\displaystyle\left[  \left( 2-\frac{v^{2}}{v_{t}^{2}}\right) ^{2}-4\sqrt{\left(
	1-\frac{v^{2}}{v_{t}^{2}}\right)  \left(  1-\frac {X_{1}^{2}}{k^{2}}\right)  }\ \right]
\sqrt{1-\frac{X_{2}^{2}}{k^{2}}}\left(
-\frac{v^{2}}{v_{s}^{2}}+\frac{X_{2}^{2}}{k^{2}}\right)  \vspace{1mm}  \notag\\
&-\displaystyle\left[  \left(  2-\frac{v^{2}}{v_{t}^{2}}\right)  ^{2}%
-4\sqrt{\left(  1-\frac{v^{2}}{v_{t}^{2}}\right)  \left(  1-\frac{X_{2}^{2}%
	}{\kappa^{2}}\right)  }\ \right] \sqrt{1-\frac{X_{1}^{2}}{k^{2}}}\left(
-\frac{v^{2}}{v_{s}^{2}}+\frac{X_{1}^{2}}{k^{2}}\right), 
\label{gensec}%
\end{align}
where
\begin{align}
X_{1,2}^{2}=X_{1,2}^{2}(\omega^{2})= & \displaystyle\frac{1}{2\alpha(\lambda_{\rm e}+2\mu_{\rm e}
	)}\Bigg\{{\beta }^{2}-(\xi -{\varrho_0 }\varkappa\omega^{2})(\lambda_{\rm e}+2\mu_{\rm e})+\alpha
\omega^{2}{\varrho_0 }\nonumber\\
&
\mp\sqrt{4\alpha\omega^{2}{\beta }^{2}{\varrho_0 }+[-{\beta }^{2}+\alpha\omega
	^{2}{\varrho_0 }+(\xi -{\varrho_0 }\varkappa\omega^{2})(\lambda_{\rm e}+2\mu_{\rm e})]^{2}}\Bigg\}.
\end{align}

In \cite{ChiritaGhiba2} the uniqueness of the solution was not demonstrated and the existence of the solution of the secular equation was proven assuming the conditions (\ref{auxasump0}). These conditions are more restrictive in comparison with the assumptions (\ref{condspeed3}). Our analysis is valid for all materials admitting planar real wave propagation and do not  enter into conflict with the necessary assumptions for classical linear elasticity. With less restrictive conditions on the constitutive parameters (connected with conditions usually imposed by the engineering community)  we have proven that the solution exists and it is unique. More than that the algorithm for computing the Rayleigh wave solution presented in this paper is clear and simple without knowing in advance the analytical expressions of some eigenvalues and their associated eigenvectors (the main difficulty to obtain a complete algorithm and the proof of existence and uniqueness in the other methods).

For the coefficients considered in this paper, the conditions \eqref{auxasump0} from \cite{ChiritaGhiba2} are not verified (${c}_{t}^{2}\leq {c}_{s}^{2}-\displaystyle\frac{{\beta }^{2}}%
{{\varrho_0 }\xi }$ is true, but \ ${c}_{t}^{2}
\leq\frac{\alpha}{{\varrho_0 } \varkappa}$ is not valid). However, as one can see in Figure \ref{method2} (the red curve) the secular equation from \cite{ChiritaGhiba2} is well defined  and it still has a unique solution in the interval $[0,\widehat{v})$. However, for the values of the constitutive coefficients considered, the existence  and uniqueness  cannot be guaranteed by the results established in \cite{ChiritaGhiba2}.
From Figure \ref{method2} we can see that the same value of the propagation speed is obtained, from both forms of the secular equation (our vs. \cite{ChiritaGhiba2}). This represents another check that the algorithm proposed by us is  viable and the computations are correct.

The same value $v_0=0.113175\,(m/s)$  is obtained as limit case in our analysis of the micro-voids model for structural steel S235 and large values of the micro-constitutive parameter  $\mu_{\rm micro}$ (i.e., a large $\xi$), e.g.,
$
\mu _{\text{micro}}=10^7 \,\mu _{\text{macro}}=10^7*161.538 \,\textrm{(GPa)},
$
which indicates the convergence of the Rayleigh wave solution from the micro-void model to the Rayleigh wave solution from classical elasticity when $\mu_{\rm micro}\to \infty$.

\bibliographystyle{plain} 

\bibliographystyle{plain} 
\addcontentsline{toc}{section}{References}

\begin{footnotesize}

\end{footnotesize}
 \end{document}